\documentclass[reqno]{amsart}
\usepackage[english]{babel}
\usepackage{amsmath,amscd,amssymb,amsthm,amsxtra}
\usepackage{microtype}
\usepackage{mathrsfs}
\usepackage{mathtools}
\usepackage{accents}
\usepackage{color,graphicx,esint}
\usepackage{epsfig,epstopdf}
\usepackage{tikz-cd}
\usetikzlibrary{decorations.pathreplacing, patterns}
\usepackage[margin=1.0in]{geometry}

\theoremstyle{plain}
\newtheorem{theorem}{Theorem}[section]

\newtheorem{proposition}[theorem]{Proposition}
\newtheorem{lemma}[theorem]{Lemma}

\theoremstyle{definition}
\newtheorem{definition}[theorem]{Definition}
\newtheorem{remark}[theorem]{Remark}

\theoremstyle{remark}
\numberwithin{theorem}{section}
\numberwithin{equation}{section}
\numberwithin{figure}{section}

\def\R{\mathbb{R}}

\def\1{{\bf 1}}
\def\e{\mathrm{e}}
\def\d{\mathrm{d}}

\def\a{\alpha}

\def\del{\delta}

\def\l{\lambda}
\def\L{\Lambda}
\def\nab{\nabla}
\def\om{\omega}

\def\vep{\varepsilon}

\def\Up{\Upsilon}

\def\pa{\partial}

\def\B{\mathcal{B}}
\def\C{\mathcal{C}}
\def\K{\mathcal{K}}

\DeclareMathOperator{\diag}{diag}

\DeclareMathOperator{\dist}{dist}

\DeclareMathOperator{\Id}{Id}

\DeclareMathOperator{\spt}{spt}
\DeclareMathOperator{\trace}{tr}

\begin{document}

\title{On the (In)stability of the Identity Map in Optimal Transportation}

\author[Y.\ Jhaveri]{Yash Jhaveri*}
\thanks{*\,Supported in part by the ERC grant ``Regularity and Stability in Partial Differential Equations (RSPDE)''}
\address{ETH Z\"{u}rich, Department of Mathematics, R\"{a}mistrasse 101, Z\"{u}rich 8092, Switzerland}
\email{yash.jhaveri@math.ethz.ch}

%~~~ABSTRACT~~~%
\begin{abstract}
We collect some examples of optimal transports in order to explore the (in)stability of the identity map as an optimal transport.
First, we consider density and domain perturbations near regular portions of domains.
Second, we investigate density and domains deformations to non-regular parts of domains.
Here, we restrict our attention to two dimensions and focus near 90 degree corners.
% \vspace{3mm}
% \noindent {\bf Keywords:} Schauder estimates
\end{abstract}
\maketitle
\vspace{-0.75cm}

%~~~MATH~~~%
\section{Introduction}

The optimal transport problem for quadratic cost asks whether or not it is possible to find a map that minimizes the total cost of moving a distribution of mass $\mu$ to another $\nu$ given the cost of moving $x$ to $y$ is measured by the squared distance between $x$ and $y$; concisely written, it is
\[
\min \bigg{\{} \int |x-T(x)|^2 \, \d\mu(x) : T_{\#} \mu = \nu \bigg{\}}.
\]
Under certain conditions on $\mu$ and $\nu$, the existence of a unique ($\mu$-a.e.) minimizing map, an optimal transport, was first discovered by Brenier in \cite{B}\,---\,he characterized optimal transports as gradients of convex functions.
The regularity of optimal maps is a delicate question and is guaranteed only under natural but strong geometric conditions.
 
Let $\mu = f(x)\,\d x$ and $\nu = g(y)\,\d y$, and set $X = \{ f > 0 \}$ and $Y = \{ g > 0 \}$, which we assume to be open, bounded subsets of $\R^n$.
If $f$ and $g$ are bounded away from zero and infinity on  $X$ and $Y$ respectively and $Y$ is convex, then Caffarelli showed, in \cite{C1}, that $u$ is a strictly convex (Alexandrov) solution to the Monge-Amp\`{e}re equation
\[
\det(D^2 u) = \frac{f}{g \circ \nab u} \quad \text{in}\quad X.
\]
From here, he developed a regularity theory for mappings with convex potentials, part of which we now recall.
Under the assumption that $Y$ is convex (\cite{C1}):
\begin{itemize}
\item[-] If $\l \leq f,g \leq 1/\l$, with $\l > 0$, then $\nab u \in C_{\rm loc}^{0,\sigma}(X)$ for some $\sigma \in (0,1)$.
\item[-] If, in addition, $f \in C^{k,\alpha}_{\rm loc}(X)$ and $g \in C^{k,\alpha}_{\rm loc}(Y)$, then $\nab u \in C^{k+1,\alpha}_{\rm loc}(X)$, for $k \geq 0$ and $\alpha \in (0,1)$.
\end{itemize}
Under the assumption that both $X$ and $Y$ are convex (\cite{C2}):
\begin{itemize}
\item[-] If $\l \leq f,g \leq 1/\l$, with $\l > 0$, then $\nab u \in C^{0,\sigma}(\overline{X})$ for some $\sigma \in (0,1)$.
\end{itemize}
Under the assumption that both $X$ and $Y$ are smooth and uniformly convex (\cite{C3}):
\begin{itemize}
\item[-] If $f \in C^{k,\alpha}(\overline{X})$ and $g \in C^{k,\alpha}(\overline{Y})$, with $f,g > 0$,  then $\nab u \in C^{k+1,\alpha}(\overline{X})$, for $k \geq 0$ and $\alpha \in (0,1)$.
\end{itemize}
That said, given any set $E$, the optimal transport taking (the constant density $1$ on) $E$ to (the constant density $1$ on) $E$ is the identity map.

In general, when we say the optimal transport taking a set $X$ to a set $Y$, we mean the optimal transport taking the density $\1_X$ to the density $ \1_Y$ (necessarily, $|X| = |Y|$.)

In this paper, we study the stability of the identity map as an optimal transport from a domain to itself.
First, we consider density and domain perturbations near regular portions of domains.
More specifically, we find an example of an arbitrarily small Lipschitz (non-convex) perturbation of (a side of) a square that, when taken as the target domain in the optimal transport problem from that square, yields a discontinuous optimal transport.
Second, noticing that the discontinuity of optimal transports is an open condition, we find that given any $\vep > 0$, there exists an $\alpha > 0$ and an $\vep$-small $C^{1,\alpha}$ perturbation of a square that produces a discontinuous optimal transport.
We then show this is sharp, via an $\vep$-regularity theorem at the boundary (in $n \geq 2$ dimensions), in the sense that given any $\alpha > 0$, there exists an $\vep > 0$ such that any $\vep$-small $C^{1,\alpha}$ perturbation of a $C^{1,\alpha}$ domain has a continuous optimal transport.
Second, we investigate density and domains deformations around non-regular parts of domains.
Here, we restrict our attention to two dimensions and focus near 90 degree corners.
We observe that the $\vep$-regularity theorem we proved on domains ``comparable'' to half balls can be extended to domains ``comparable'' to quarter discs.
Finally, we show that given two smooth densities $f$ and $g$ on the unit square, the optimal transport taking $f$ to $g$ is of class $C^{2,\alpha}$ up to the boundary for every $\alpha < 1$, yet it may not be of class $C^3$, even with densities that are arbitrarily $C^\infty$-close to $1$.

%~~~~~~~~~~~~~~~~~~~~~~~~~~~~~~~~~~~~~~~~~~~~~~~~~~~~~~~~~%
\section{Perturbations in Regular Domains}

In \cite{C1}, Caffarelli showed that the optimal transport $\nab u_\vep$ taking the ball $B_1 \subset \R^2$ to the dumbbell $D_\vep := (B_{r_\vep}^{+} + \e_1) \cup (B_{r_\vep}^{-} -\e_1) \cup ([-1,1] \times ( -\vep, \vep))$, where $r_\vep > 0$ is taken so that $|D_\vep| = |B_1|$, is discontinuous for all $\vep > 0$ sufficiently small.\footnote{\,In actuality, he showed that the optimal transport from $B_1$ to a smoothing of $D_\vep$ is discontinuous.
However, the regularity of $D_\vep$ is irrelevant to the essence of the singular nature of his example.}
Here, $B_{r}^{+} := B_{r} \cap \{ x_1 > 0 \}$ and $B_{r}^{-} := B_{r} \cap \{ x_1 < 0 \}$ for $r > 0$.
His example demonstrates the importance of having a convex target in guaranteeing an optimal transport's regularity.
A natural follow-up question is, how important is the convexity of the target space in guaranteeing the regularity of the optimal transport?
We shall see that even a small deviation from convexity can break the continuity of an optimal transport.

Before presenting our examples, let us review Caffarelli's example.
Notice that the optimal transport taking $B_1$ to $D_0 := (B_1^{+} + \e_1) \cup (B_1^{-} -\e_1)$ is given by
\[
\nab u_0(x) =
\begin{cases}
x + \e_1 &\text{if } x_1 > 0\\
x - \e_1 &\text{if } x_1 < 0.
\end{cases} 
\]
By the stability of optimal transports, up to the addition of constants, the potentials $u_\vep$ converge locally uniformly (in $\R^2$) to $u_0$.
So, since $|\pa u_0(\{ (0,\pm 1) \})| = 0$, it follows that
\[
\lim_{\vep \to 0} |\pa u_\vep(\{0\} \times (-1,1))| = |\pa u_0(\{0\} \times (-1,1))| = 4. 
\]
In turn, we see that the Monge-Amp\`{e}re measure associated to $u_\vep$ must have a singular part for all $\vep > 0$ sufficiently small, that is, $\nab u_\vep$ is discontinuous for all $\vep > 0$ sufficiently small.
(For a more a hands on explication of Caffarelli's example, one that appeals to the monotonicity of optimal transports, we refer the reader to \cite{CJLPR}.) 

\subsection{Lipschitz Perturbations.}
\label{sec: Lip pert}
Here, we present an example of an $\vep$-Lipschitz perturbation of a square that after smoothing proves the following:

\begin{theorem}
\label{thm: Lip discty}
Given any $\vep > 0$, there exists a smooth, convex domain $X$ and a domain $Y$ that is an $\vep$-small Lipschitz perturbation of $X$ such that the optimal transport taking $X$ to $Y$ is discontinuous.
\end{theorem}

\begin{proof}
Let 
\[
X : = (0,4) \times (-2,2) \qquad\text{and}\qquad Y : = \{ (0,4+ \vep/4) \times (-2,2) \} \setminus \overline{\Gamma}_{\vep}
\]
where $\Gamma_\vep$ is the interior of the triangle with vertices $(\vep,0)$, $(0,1)$, and $(0,-1)$.
Notice that $X$ and $Y$ have the same volume and $Y$ is an $\vep$-Lipschitz (non-convex) perturbation of $X$.
If $T = \nab u$ is the optimal transport taking $Y$ to $X$, then from \cite{C1,FL}, we have that $u$ is a strictly convex Alexandrov solution of
\[
\det(D^2 u) = 1\quad\text{in}\quad Y
\] 
and of class $C^\infty_{\rm loc}(Y) \cap C^1(\R^2)$.

Let $X'$ and $Y'$ be the reflections of $X$ and $Y$ over the lines $\{x_2=2\}$ and $\{y_2 = 2 \}$ respectively and $T'$ be the optimal transport taking $Y'$ to $X'$.
Then, the map $S(y) := RT'(Ry)$ where $R$ is the reflection over the line $\{y_2 = 2 \}$ (and also over the line $\{x_2 = 2 \}$) is a competing transport map with equal cost.
So, $S = T'$.
Moreover, $T'|_{\overline{Y}} = T$.
It follows that $T(\{y_2 = 2\}) \subset \{x_2 = 2\}$.
Similarly, considering reflections of $X$ and $Y$ over the lines $\{x_1 = 4\}$ and $\{y_1 = 4 + \vep/4\}$ and the lines $\{x_2=-2\}$ and $\{y_2 = -2\}$, reflections of $X^+ := X \cap \{ x_2 > 0 \}$ and $Y^+ := Y \cap \{ y_2 > 0 \}$ over the lines $\{x_1=0\}$ and $\{y_1=0\}$, and reflections of $X^- := X \cap \{ x_2 < 0 \}$ and $Y^- := Y \cap \{ y_2 < 0 \}$ over the lines $\{x_1=0\}$ and $\{y_1=0\}$, we deduce that $T = \nab u$ maps $\{4 +\vep/4\} \times [-2,2]$ homeomorphically to $\{4\} \times [-2,2]$, maps $[0,4+\vep/4] \times \{\pm 2\}$ homeomorphically to $[0,4] \times \{\pm 2\}$, and maps $\{0\} \times [1,2]$ and $\{0\} \times [-2,-1]$ homeomorphically to subsegments of $\{0\} \times [0,2]$ and $\{0\} \times [-2,0]$ respectively. 
Also, by symmetry and restriction, $u$ is strictly convex on $\overline{Y} \cap \{ y_2 \geq 0\}$ and $\overline{Y} \cap \{ y_2 \leq 0\}$ (see \cite{C2}).

There are two possibilities. 
Either $\nab u(\pa Y) = \pa X$ and $u$ is strictly convex on $\overline{Y}$ or some portion (symmetric with respect to the $y_1$-axis) of the left boundary of $Y$ will map inside $X$.
In particular, in the second scenario, a symmetric subset of the two segments joining $(0,1)$, $(\vep,0)$, and $(0,-1)$ and containing the point $(\vep,0)$ will map to a segment along $X \cap \{ x_2 = 0 \}$, and the optimal transport $\nab u^\ast$, where $u^\ast$ is the Legendre transform of $u$, taking $X$ to $Y$ will be discontinuous along the segment joining $(0,0)$ and $(0,t_\vep)$ where $\nab u(\vep,0) = (0,t_\vep)$.
(See Figure~\ref{fig: discty}.)

Suppose that $\nab u(\pa Y) = \pa X$.
Then, taking the partial Legendre transform of $u$ in the $\e_1$-direction and setting
$v = \pa_1 u^\star$, we find that
\[
\begin{cases}
\Delta v = 0 &\text{in } X\\
v = h(p) &\text{on } (0,4) \times \{\pm 2\}\\
v = 4 + \vep/4 &\text{on } \{4\} \times (-2,2)\\
v = \max \{ 0, -\vep|x_2|+\vep \} &\text{on }\{0\} \times (-2,2)
\end{cases}
\]
where $h$ is an increasing function such that $h(0) = 0$ and $h(4) = 4+ \vep/4$.
(See the proof of Theorem~\ref{thm: C2a up to bdry Q} for details on the partial Legendre transform.)
Consider the harmonic function
\[
b(p,x_2) := \vep\frac{2}{\pi}\Re(z \log(z)) + \vep + 2(x_2^2 - p^2) + 16p
\]
where $z = p + i x_2$ and $\Re(z \log(z))$ denotes the real part of $z \log(z)$.
Observe that $b$ is an upper barrier for $v$.
Thus, as $b(0) = v(0)$ and $b(p,0) - b(0,0) < 0$ for all $p > 0$ sufficiently small, we deduce that
\[
\int_0^{p} \pa_1 v(t,0) \, \d t < 0
\]
for all $0 < p \ll 1$.
It follows that
\[
\pa_{11}u^\star(p,0) = \pa_1 v(p,0) < 0
\]
for all sufficiently small $p > 0$.
But, this contradicts the convexity of $u^\star$ in the $\e_1$-direction.
\end{proof}

\begin{figure}[ht]
\label{fig: discty}
\begin{tikzpicture}[line cap=round,line join=round,scale=2.5,baseline={(0,0)}]
\pgfmathsetmacro{\eps}{0.2}
\path[fill=black!5](\eps,0)--(0,1/2)--(0,1)--(2+\eps/4,1)--(2+\eps/4,0)--cycle;
\draw(0,-1)--(0,-1/2)--(\eps,0)--(0,1/2)--(0,1)--(2+\eps/4,1)--(2+\eps/4,-1)--cycle;
\draw[dashed, line width=0.01pt](-\eps/2,0)--(2+\eps/4+\eps/2,0);
\draw[decorate,decoration={brace},yshift=0.5pt](0,1/2)--(\eps,1/2) node[midway,above] {$\vep$};
\draw[decorate,decoration={brace},xshift=-0.5pt](0,0)--(0,1/2) node[midway,left] {$1$};
\path[fill=black!100](\eps/2,-1/4)--(\eps/2,1/4)--(\eps,0);
\node (Y) at (1+\eps/2,-1) {};
\draw (Y) node[below] {$Y$};
\path[fill=black!5](3,0)--(3,1)--(5,1)--(5,0)--cycle;
\draw(3,-1)rectangle(5,1);
\draw[dashed, line width=0.01pt](3-\eps/2,0)--(5+\eps/2,0);
\draw[line width=1pt, line cap=butt](3,0)--(3+\eps,0);
\draw[decorate,decoration={brace,mirror},yshift=-0.5pt](3,0)--(3+\eps,0) node[midway,below] {$t_\vep$};
\node (X) at (4,-1) {};
\draw (X) node[below] {$X$};
\draw[->] (3-\eps,-1/2)--(2+\eps,-1/2)node[midway,fill=white,inner sep=1pt]{$\nab u^\ast$};
\draw[->] (2+\eps,1/2)--(3-\eps,1/2)node[midway,fill=white,inner sep=1pt]{$\nab u$};
\draw[->,line width=0.01pt] (3,0) to[bend left=20,looseness=1] (\eps/2,-1/4);
\draw[->,line width=0.01pt] (3,0) to[bend right=20,looseness=1] (\eps/2,1/4);
\draw[->] (3+\eps,0) to[bend right=30,looseness=1] (\eps,0);
\end{tikzpicture}
\caption{The optimal transport $\nab u^\ast$ splits mass.}
\end{figure}
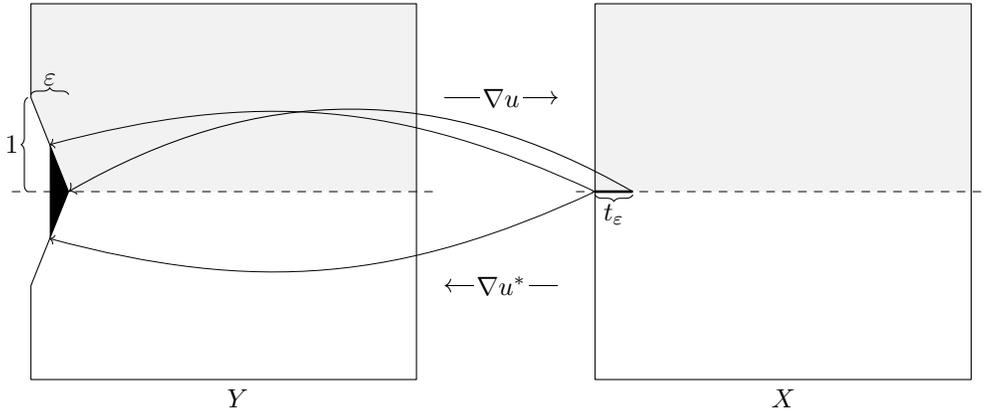

\subsection{$C^{1,\alpha}$ Perturbations.}
Notice that the proof of Caffarelli's optimal transport's discontinuity shows that the discontinuity of optimal transports is an open condition.
More precisely, by the stability of optimal transports, we have the following lemma:

\begin{lemma}
\label{lem: open}
Let $X_0,Y_0 \subset \R^n$ be open, bounded sets and $X_\vep, Y_\vep \subset \R^n$ be a sequences of open, bounded sets such that $\dist(\pa X_\vep,\pa X_0)+\dist(\pa Y_\vep,\pa Y_0) \to 0$ as $\vep \to 0$.
Let $f_\vep$ and $g_\vep$ be sequences of densities uniformly bounded away from zero and infinity on $X_\vep$ and $Y_\vep$ respectively, satisfying the mass balance condition $\|f_\vep\|_{L^1(X_\vep)} = \|g_\vep\|_{L^1(Y_\vep)}$, and such that $f_\vep \to f_0$ and $g_\vep \to g_0$ in $L^1$ as $\vep \to 0$.
If the optimal transport $\nab u_0$ taking $f_0$ to $g_0$ is discontinuous, then there exists an $\vep_0 > 0$ such that the optimal transport $\nab u_\vep$ taking $f_\vep$ to $g_\vep$ is discontinuous for all $\vep < \vep_0$.
\end{lemma}

Lemma~\ref{lem: open} allows us to extend Theorem~\ref{thm: Lip discty} to small $C^{1,\alpha}$ perturbations.

\begin{theorem}
\label{thm: C1a discty}
Given any $\vep > 0$, there exists an $\alpha = \alpha(\vep) > 0$, a smooth, convex domain $X$, and a domain $Y$ that is an $\vep$-small $C^{1,\alpha}$ perturbation of $X$ such that the optimal transport taking $X$ to $Y$ is discontinuous.
\end{theorem}

That said, $C^{1,\alpha}$ is the borderline topology in which small perturbations can break an optimal transport's continuity.
More precisely, Theorem~\ref{thm: C1a discty} is sharp in view of Theorem~\ref{thm: C1a cty}.

\begin{theorem}
\label{thm: C1a cty}
Let $X$ be a $C^{1,\alpha}$ domain, $Y$ be an $\vep$-small $C^{1,\alpha}$ perturbation of $X$, and $\nab u$ be the optimal transport taking $X$ to $Y$.
If $\vep = \vep(\alpha) > 0$ is sufficiently small, then $\nab u : \overline{X} \to \overline{Y}$ is a bi-H\"{o}lder continuous homeomorphism.
\end{theorem}

\begin{proof}
The proof of Theorem~\ref{thm: C1a cty} is a consequence of a simple localization procedure (a pair of translations, a rotation, a shearing (affine) transformation, and a dilation), Theorem~\ref{thm: c1b at bdry}, and a covering argument.
\end{proof}

The following theorem is a generalization of the $\vep$-regularity theorem at the boundary for $C^2$ domains proved by Chen and Figalli (\cite[Theorem~2.1]{CF}) to $C^{1,\alpha}$ domains.
Heuristically, since $C^{1,\alpha}$ domains, like $C^2$ domains, flatten under dilations, we might expect that extending their arguments to our setting is rather simple.
However, in practice, our situation is quite delicate and some additional details and new ideas must be presented and developed.
That said, for an explanation of any estimate or computation that does not specifically see the difference between a $C^2$ and $C^{1,\alpha}$ boundary, we refer the reader to their proof or the proofs of \cite[Theorem~4.3]{DF} and \cite[Proposition~4.2]{J}.

Let $x \in \R^n$ be given by $x = (x_0,x_n) \in \R^{n-1} \times \R$ and $\B_R := \{ x_0 \in \R^{n-1} : |x_0| < R \} = B_R \cap \{ x_n = 0 \}$.

\begin{theorem}
\label{thm: c1b at bdry}
Let $\C$ and $\K$ be two closed subsets of $\R^n$ such that
\[
B_{1/2} \cap \{x_n \geq \gamma(x_0) \} \subset \C \subset B_2 \cap \{x_n \geq \gamma(x_0) \}
\]
and
\[
B_{1/2} \cap \{y_n \geq \zeta(y_0) \} \subset \K \subset B_2 \cap \{y_n \geq \zeta(y_0) \}
\]
where 
\begin{equation}
\label{eqn: bdry 0 at 0}
\gamma,\zeta \in C^{1,\alpha}(\B_4), \qquad \gamma(0)=\zeta(0)= 0, \qquad\text{and}\qquad \nab \gamma(0)=\nab \zeta(0)= 0.
\end{equation} 
Let $u$ be a convex potential such that $(\nab u)_{\#}f = g$ for two densities $f$ and $g$ supported on $\C$ and $\K$ respectively.
Given $\beta \in (0,1)$, there exist constants $r, \eta, \delta > 0$, with $\delta = \delta(\eta)$ and $\eta = \eta(\alpha,\beta,n)$, such that the following holds: if
\begin{equation}
\label{eqn: bdry data is flat}
[\nab \gamma]_{C^{0,\alpha}(\B_4)} + [\nab \zeta]_{C^{0,\alpha}(\B_4)} \leq \delta,
\end{equation}
\begin{equation}
\label{eqn: densities close to 1}
\|f - \1_\C\|_{\infty} + \|g - \1_\K\|_{\infty} \leq \del,
\end{equation}
and
\begin{equation}
\label{eqn: solution close to quadratic}
\bigg{\|}u(x)-\frac{1}{2}|x|^2\bigg{\|}_{L^\infty(B_{1/2} \cap \{x_n \geq \gamma(x_0) \})} \leq \eta,
\end{equation}
then $u \in C^{1,\beta}(B_r \cap \{x_n \geq \gamma(x_0) \})$.
\end{theorem}

In what follows, we let $C$ and $c$ be generic positive constants that may change from line to line.
Their dependencies, if any, will either be clear from context or explicitly given.

Before proceeding with the proof of Theorem~\ref{thm: c1b at bdry}, let us make a remark and an associated definition.
From the point of view of optimal transportation, the cost $-x\cdot y$ is the same as $-x\cdot y|_{\spt f \times \spt g}$.
So, we shall often work with the intersection of the subdifferential of our convex potential with the support of our target measure.
In particular, we define
\[
\pa_\ast u(x) := \pa u(x) \cap \spt g \qquad\text{and}\qquad \pa_\ast u(E) := \bigcup_{x \in E} \pa_\ast u(x)
\]
when $(\nab u)_{\#} f = g$.

\begin{proof}
For clarity's sake, we divide the proof into several steps.\\

{\it -- Step 1: An initial normalization.}\\

Let 
\[
x^0 \in B_r \cap \{x_n \geq \gamma(x_0) \}  \qquad\text{and}\qquad y^0 \in \pa_\ast u(x^0).
\]  
Thus, using \eqref{eqn: bdry 0 at 0} and \eqref{eqn: bdry data is flat}, we have that
\begin{equation}
\label{eqn: lower bound ybar2}
y^0_n \geq \zeta(y^0_0) \geq -4\delta.
\end{equation}

Notice that
\[
w(x) := u(x) - \frac{1}{2}|x|^2 + \frac{1}{2}|x-x^0|^2
\]
is convex and $y^0-x^0 \in \pa w(x^0)$.
Hence, by \eqref{eqn: solution close to quadratic}, \eqref{eqn: bdry 0 at 0}, and \eqref{eqn: bdry data is flat}, we deduce that for $\e \in \mathbb{S}^{n-1}$ such that $\angle(\e, \e_n) = \pi/4$,
\begin{equation}
\label{eqn: bar y - bar x estimate}
(y^0 - x^0) \cdot \e \leq \frac{w(x^0 + \eta^{1/2} \e) - w(x^0)}{\eta^{1/2}} \leq \frac{5}{2}\eta^{1/2}
\end{equation}
provided that $r+\eta^{1/2} < 1/2$.
The same estimate holds with $\e = \e_n$.
If $x^0_n - \gamma(x^0_0) \geq \eta^{1/2}$, then $x^0 - \eta^{1/2}\e_n \in B_{1/2} \cap \{x_n \geq \gamma(x_0) \}$.
So, using $w$ as before, we find that
\[
(y^0 - x^0)  \cdot (-\e_n) \leq \frac{5}{2}\eta^{1/2}.
\]
If, on the other hand, $x^0_n - \gamma(x^0_0) < \eta^{1/2}$, then
\[
(y^0 - x^0)  \cdot (-\e_n)  
\leq \eta^{1/2} + (4+r)\delta
\leq C(\eta^{1/2}+\delta),
\]
recalling \eqref{eqn: lower bound ybar2} and noticing that $|\gamma(x^0_0)| \leq \delta r$ by \eqref{eqn: bdry 0 at 0} and \eqref{eqn: bdry data is flat}.
As every $\theta \in \mathbb{S}^{n-1}$ can be written as a linear combination of $\pm \e_n$ and some $\e$ such that $\angle(\e, \e_n) = \pi/4$ with positive coefficients, it follows that
\begin{equation}
\label{eqn: image of point under subdiff}
|y^0 - x^0| \leq C(\eta^{1/2}+\delta).
\end{equation}

First, consider the change of variables
\[
\hat{x} := x - x^0 \qquad\text{and}\qquad \hat{y} := y - y^0.
\]
Notice that
\[
\hat{\gamma}(\hat{x}_0) := \gamma(\hat{x}_0+x^0_0) - x^0_n \qquad\text{and}\qquad \hat{\zeta}(\hat{y}_0) := \zeta(\hat{y}_0+y^0_0) - y^0_n
\] 
define the lower boundaries of $\hat{\C} := \C - x^0$ and $\hat{\K}:= \K - y^0$ respectively and, by \eqref{eqn: bdry 0 at 0}, \eqref{eqn: bdry data is flat}, and \eqref{eqn: image of point under subdiff},
\[
0 \geq \hat{\zeta}(0) > -C(\eta^{1/2} + \delta).
\]
Furthermore, using \eqref{eqn: solution close to quadratic} and \eqref{eqn: image of point under subdiff}, we have that
\begin{equation*}
\label{eqn: translated solution close to quadratic}
\bigg{\|}\hat{u}(\hat{x})-\frac{1}{2}|\hat{x}|^2\bigg{\|}_{L^\infty(B_{5/12} \cap \{ \hat{x}_n \geq \hat{\gamma}(\hat{x}_0) \})} \leq C(\eta^{1/2}+\delta)
\end{equation*}
where
\[
\hat{u}(\hat{x}) := u(x) - u(x^0) - y^0\cdot(x-x^0).
\]
Also, defining $\hat{f}(\hat{x}) := f(\hat{x}+x^0)$ and $\hat{g}(\hat{y}) := g(\hat{y}+y^0)$, it is clear that $(\nab \hat{u})_\# \hat{f} = \hat{g}$.
\\

{\it -- Case 1: $x^0 \in \{ x_n = \gamma(x_0)\}$.}\\

Let $R$ be the rotation matrix that makes the tangent line to $\hat{\C}$ at $0$ (which is on the lower boundary of $\hat{\C}$) horizontal and consider the change of coordinates
\[
\bar{x} := R\hat{x} \qquad\text{and}\qquad \bar{y} := (R^\ast)^{-1}\hat{y}.
\]
By \eqref{eqn: bdry 0 at 0} and \eqref{eqn: bdry data is flat}, we see that $|\nab \hat{\gamma}(0)| \leq \delta r^\alpha$.
Therefore, the angle defining $R$ is smaller than $\delta r^\alpha$.
So, letting $\bar{\gamma}$ and $\bar{\zeta}$ define the lower boundaries of $\bar{\C} := R\hat{\C}$ and $\bar{\K} := (R^\ast)^{-1}\hat{\K}$ respectively, it follows that
\[
[\nab \bar{\gamma}]_{C^{0,\alpha}(\B_3)} + [\nab \bar{\zeta}]_{C^{0,\alpha}(\B_3)} \leq C\delta
\]
and
\[
0 \geq \bar{\zeta}(0) > -C(\eta^{1/2} + \delta).
\]
Furthermore, letting
\[
\bar{u}(\bar{x}) := \hat{u}(R^{-1}\bar{x}), \qquad \bar{f}(\bar{x}) := \hat{f}(R^{-1}\bar{x}), \qquad\text{and}\qquad \bar{g}(\bar{y}) := \hat{g}(R^\ast\bar{y}),
\]
we have that
\begin{equation*}
\bigg{\|}\bar{u}(\bar{x})-\frac{1}{2}|\bar{x}|^2\bigg{\|}_{L^\infty(B_{5/12} \cap \{ \bar{x}_n \geq \bar{\gamma}(\bar{x}_0) \})} \leq C(\eta^{1/2}+\delta)
\end{equation*}
and $(\nab \bar{u})_\# \bar{f} = \bar{g}$.

Finally, define the change of variables
\[
\check{x} := N\bar{x} \qquad\text{and}\qquad \check{y} := (N^{\ast})^{-1} \bar{y}
\]
with
\[
N\bar{z} := \bar{z} + (\nab \bar{\zeta}(0),0)\bar{z}_n.
\]
If we let $\check{\gamma}$ and $\check{\zeta}$ define the lower boundaries of $\check{\C} := N\bar{\C}$ and $\check{\K} := (N^\ast)^{-1}\bar{\K}$ respectively, then 
\[
\nab \check{\gamma}(0) = \nab \check{\zeta}(0) = 0.
\]
Moreover,
\[
[\nab \check{\gamma}]_{C^{0,\alpha}(\B_3)} + [\nab \check{\zeta}]_{C^{0,\alpha}(\B_3)} \leq C\delta.
\]
Also, note that
\[
|N - \Id| \leq C\delta
\]
provided $\eta$, $\delta$, and $r$ are sufficiently small.
\\ 

{\it -- Case 2: $x^0 \in \{ x_n > \gamma(x_0)\}$.}\\

From \eqref{eqn: bdry 0 at 0} and \eqref{eqn: bdry data is flat}, we see that the angle between the $\hat{x}_n$-axis and the line through the origin that meets $\{ \hat{x}_n = \hat{\gamma}(\hat{x}_0) \}$ orthogonally is at most $4\delta$.
So, let $R$ be the rotation matrix that makes this line vertical and consider the change of coordinates
\[
\check{x} := R\hat{x} \qquad\text{and}\qquad \check{y} := (R^\ast)^{-1}\hat{y}.
\]
Then,
\[
[\nab \check{\gamma}]_{C^{0,\alpha}(\B_3])} + [\nab \check{\zeta}]_{C^{0,\alpha}(\B_3)} \leq C\delta
\]
where $\check{\gamma}$ and $\check{\zeta}$ define the lower boundaries of $\check{\C} := R\hat{\C}$ and $\check{\K} := (R^\ast)^{-1}\hat{\K}$ respectively.
\\

In summary, if we define the potential
\[
\check{u}(\check{x}) := \bar{u}(N^{-1}\check{x}) \qquad\text{or}\qquad \check{u}(\check{x}) := \hat{u}(R^{-1}\check{x})
\]
and the densities
\[
\check{f}(\check{x}) := \bar{f}(N^{-1}\check{x}) \quad\text{and}\quad \check{g}(\check{y}) := \bar{g}(N^\ast \check{y}) \qquad\text{or}\qquad \check{f}(\check{x}) := \hat{f}(R^{-1}\check{x}) \quad\text{and}\quad \check{g}(\check{y}) := \hat{g}(R^\ast \check{y})
\]
depending on whether we are in Case 1 or Case 2, then
\[
(\nab \check{u})_{\#}\check{f} = \check{g},
\]
and provided that $r$, $\delta$, and $\eta$ are sufficiently small,
\[
B_{1/3} \cap \{\check{x}_n \geq \check{\gamma}(\check{x}_0) \} \subset \check{\C} \subset B_3 \cap \{\check{x}_n \geq \check{\gamma}(\check{x}_0) \}
\]
and
\[
B_{1/3} \cap \{\check{y}_n \geq \check{\zeta}(\hat{y}_0) \} \subset \check{\K} \subset B_3 \cap \{\check{y}_n \geq \check{\zeta}(\check{y}_0) \}
\]
with
\[
[\nab \check{\gamma}]_{C^{0,\alpha}(\B_3)} + [\nab \check{\zeta}]_{C^{0,\alpha}(\B_3)} \leq \check{\delta}.
\]
If $x^0_n = \gamma(x^0_0)$, then 
\[
\check{\gamma}(0) = 0, \qquad -(\check{\eta}+\check{\delta}) \leq \check{\zeta}(0) \leq 0, \qquad\text{and}\qquad \nab \check{\gamma}(0) = \nab \check{\zeta}(0)=0;
\]
where as if $x^0_n > \gamma(x^0_0)$, then
\[
-r < \check{\gamma}(0) \leq 0, \qquad -(\check{\eta}+\check{\delta} + r) < \check{\zeta}(0) \leq 0, \qquad \nab\check{\gamma}(0)=0, \qquad\text{and}\qquad |\nab \check{\zeta}(0)| \leq \check{\delta}.
\]
Furthermore,
\begin{equation}
\label{eqn: check u close to parabola}
\bigg{\|}\check{u}(\check{x})-\frac{1}{2}|\check{x}|^2\bigg{\|}_{L^\infty(B_{1/3} \cap \{ \check{x}_n \geq \check{\gamma}(\check{x}_0) \})} \leq \check{\eta},
\end{equation}
and by \eqref{eqn: densities close to 1},
\[
\|\check{f} - \1_{\check{\C}}\|_{\infty} + \|\check{g} - \1_{\check{\K}}\|_{\infty} \leq \check{\del}.
\]
Here, $\check{\delta} \to 0$ as $\delta \to 0$ and $\check{\eta} \to 0$ as $\eta^{1/2} + \delta \to 0$.
\\

{\it -- Step 2:  Finding and estimating a smooth approximation of $\check{u}$.}\\

We begin with an important lemma.

\begin{lemma}
\label{lem: cptness}
Let $\C$ and $\K$ be two closed subsets of $\R^n$ such that
\[
B_{1/R} \cap \{x_n \geq \gamma(x_0) \} \subset \C \subset B_R \cap \{x_n \geq \gamma(x_0) \}
\]
and
\[
B_{1/R} \cap \{y_n \geq \zeta(y_0) \} \subset \K \subset B_R \cap \{y_n \geq \zeta(y_0) \}
\]
where $\gamma$ and $\zeta$ are of class $C^{1,\alpha}(\B_R)$ and such that
\[
-\frac{1}{R^3} \leq \gamma(0), \zeta(0) \leq 0 \qquad\text{and}\qquad |\nab \gamma(0)|, |\nab \zeta(0)| \leq \delta.
\]
Let
\[
l_\gamma := |\gamma(0)| + |\nab \gamma(0)|R + [\nab \gamma]_{C^{0,\alpha}(\B_R)}R^2 \qquad\text{and}\qquad
l_\zeta :=  |\zeta(0)|  + |\nab \zeta(0)|R + [\nab \zeta]_{C^{0,\alpha}(\B_R)}R^2,
\]
and define
\[
\C_+ := \C \cup (B_{1/R} \cap \{ x_n \geq  -l_\gamma \}) \qquad\text{and}\qquad
\K_+ := \K \cup (B_{1/R} \cap \{ y_n \geq  -l_\zeta \}).
\]
Suppose $u$ is a convex function such that $(\nab u)_{\#}f = g$ for two densities $f$ and $g$ supported on $\C$ and $\K$ respectively. 
Set $\l > 0$ be such that $|\C_+| = |\l \K_+|$, where $\l \K_+$ denotes the dilation of $\K_+$ with respect to the origin, and let $v$ be a convex function such that $v(0) = u(0)$ and $(\nab v)_{\#}\1_{\C_+} = \1_{\l \K_+}$.
In addition, let $u^\ast$ and $v^\ast$ be such that $u^\ast(0) = v^\ast(0)$, $(\nab u^\ast)_{\#}g = f$, and $(\nab v^\ast)_{\#}\1_{\l\K_+} = \1_{\C_+}$.
Then, there exists a nonnegative, increasing function $\om = \om(\del)$, depending on $R$, such that $\om(\del) \geq \del$, $\om(0^+) = 0$, and the following holds: if
\[
[\nab \gamma]_{C^{0,\alpha}(\B_R)} + [\nab \zeta]_{C^{0,\alpha}(\B_R)} \leq \delta
\]
and
\[
\|f - \1_\C\| + \|g - \1_\K\| \leq \del,
\]
then
\[
\|u-v\|_{L^\infty(B_{1/R} \cap \C)}  + \|u^\ast-v^\ast\|_{L^\infty(B_{1/R^2} \cap \K)} \leq \om(\delta).
\]
\end{lemma}

\begin{proof}[Proof of Lemma~\ref{lem: cptness}]
The proof is identical to that of \cite[Lemma~4.1]{CF}.
\end{proof}

Choose $\l > 0$ such that $|\check{\C}_+| = |\l \check{\K}_+|$, where $\check{\C}_+$ and $\check{\K}_+$ are defined as in Lemma~\ref{lem: cptness}, and let $\check{v}$ be a convex function such that $(\nab \check{v})_{\#} \1_{\check{\C}_+} = \1_{\l \check{\K}_+}$ and $\check{v}(0) = \check{u}(0) = 0$.
By Lemma~\ref{lem: cptness},
\begin{equation}
\label{eqn: tilde u close to tilde v}
\|\check{u}-\check{v}\|_{L^\infty(B_{1/3} \cap \{\check{x}_n \geq \check{\gamma}(\check{x}_0) \})} \leq \om(\check{\del}).
\end{equation}
Define $\check{\C}_+'$ to be the reflection of $\check{\C}_+$ over the hyperplane $\{\check{x}_n = -l_{\check{\gamma}} \}$ and $(\l \check{\K}_+)'$ to be the reflection of $\l \check{\K}_+$ over the hyperplane $\{\check{y}_n = -\l l_{\check{\zeta}} \}$.
If $\check{v}'$ is a convex potential whose gradient is the optimal transport taking $\1_{\check{\C}_+'}$ to $\1_{(\l \check{\K}_+)'}$, then, by symmetry, $\nab \check{v}'|_{\check{\C}_+} = \nab \check{v}$.
Also,
\begin{equation}
\label{eqn: symmetry}
\nab \check{v}'( \{\check{x}_n = -l_{\check{\gamma}} \} ) \subset \{\check{y}_n = -\l l_{\check{\zeta}} \}.
\end{equation}
Without loss of generality, $\check{v}'(0) = \check{v}(0) = 0$.
% So, $\check{v}'$ is the even reflection of $\check{v}$ over the hyperplane $\{\check{x}_n = -l_{\check{\gamma}} \}$.
Therefore, by \eqref{eqn: check u close to parabola} and \eqref{eqn: tilde u close to tilde v}, 
it follows that
\[
\bigg{\|}\check{v}'(\tilde{x})-\frac{1}{2}|\check{x}|^2\bigg{\|}_{L^\infty(B_{1/3} \cap \{\check{x}_n \geq \check{\gamma}(\check{x}_0) \})} 
\leq \check{\eta} + \om(\check{\del}).
\]
Then, by symmetry and the convexity of $\check{v}'$, we deduce that
\begin{equation}
\label{eqn: convexity for reg}
\bigg{\|}\check{v}'(\tilde{x})-\frac{1}{2}|\check{x}|^2\bigg{\|}_{L^\infty(B_{1/6})} 
\leq C_0(\check{\eta} + \om(\check{\del}) + |\check{\gamma}(0)| + |\nab \check{\gamma}(0)|+|\check{\zeta}(0)| + |\nab \check{\zeta}(0)|).
\end{equation}
So, provided $\check{\eta}$, $\check{\delta}$, and $r$ are sufficiently small (recall the lines just before \eqref{eqn: check u close to parabola}), arguing as in \cite[Theorem~4.3, Step 1]{DF}, we have that 
\begin{equation}
\label{eqn: MA equation for v tilde}
\det(D^2\check{v}') = 1 \quad\text{in}\quad B_{1/7},
\end{equation}
in the Alexandrov sense, and $\check{v}'$ is uniformly convex and smooth inside $B_{1/8}$.
In particular,
\begin{equation}
\label{eqn: C3 estimate}
\|\check{v}'\|_{C^3(B_{1/8})} \leq C_1 \qquad\text{and}\qquad \frac{1}{C_1} \Id \leq D^2\check{v}' \leq C_1 \Id \quad\text{in}\quad B_{1/8}.
\end{equation}
Moreover,
\begin{equation}
\label{eqn: sections of tilde v}
Z_{\check{v}'}(h) := \{ \check{x} : \check{v}'(\check{x}) \leq \nab \check{v}'(0)\cdot \check{x} + h \} \Subset B_{1/9} \qquad\forall h \leq \check{h}
\end{equation}
where $\check{h} > 0$ is a small constant.

From now on, we will not distinguish $\check{v}'$ and $\check{v}$.

Let us estimate $\nab \check{v}(0)$ and $D^2 \check{v}(0)$.
Arguing as we did to prove \eqref{eqn: image of point under subdiff}, considering the convex function
\[
\check{w}(\check{x}) := \check{u}(\check{x}) - \check{v}(\check{x}) + \frac{C_1}{2}|\check{x}|^2,
\]
we deduce that
\begin{equation}
\label{eqn: gradient tilde v at 0}
|\nab \check{v}(0)| \leq C_2\om(\check{\del})^{1/2}.
\end{equation}
By \eqref{eqn: symmetry}, $\pa_n\check{v}$ is constant on $\{\check{x}_n = -l_{\check{\gamma}}  \}$, from which we infer that
\begin{equation}
\label{eqn: symmetry Hessian at 0}
\pa_{in} \check{v}|_{\{\check{x}_n = -l_{\check{\gamma}} \}} \equiv 0 \qquad \forall i = 1, \dots, n-1.
\end{equation}
Also, using \eqref{eqn: C3 estimate}, we find that
\begin{equation}
\label{eqn: Hessian of tilde v at 0}
|\pa_{ij}\check{v}(0) - \pa_{ij}\check{v}(0,-l_{\check{\gamma}})| \leq C_1l_{\check{\gamma}}.
\end{equation}
\

{\it -- Step 3: Estimating the eccentricity of $\check{u}$.}\\

Let 
\[
A := [D^2\check{v}(0)]^{-1/2}.
\]
Taylor expanding $\check{v}$ around the origin, recalling that $\check{v}(0) = \check{u}(0) = 0$, and using \eqref{eqn: tilde u close to tilde v}, \eqref{eqn: C3 estimate}, and \eqref{eqn: gradient tilde v at 0}, we see that
\begin{equation}
\label{eqn: intial poly approximation of tilde u}
\begin{split}
\bigg{\|} \check{u}(\check{x}) - \frac{1}{2}|A^{-1}\check{x}|^2\bigg{\|}_{L^\infty(AB_{(16h)^{1/2}} \cap \{\check{x}_n \geq \check{\gamma}(\check{x}_0) \})} 
&\leq \om(\check{\del}) + 4C_1^{1/2}C_2\om(\check{\del})^{1/2}h^{1/2} + 64C_1^{5/2}h^{3/2}\\ 
&\leq \frac{1}{2}\check{\eta}h
\end{split}
\end{equation}
first choosing $h$ and then choosing $\check{\del}$ also sufficiently small depending on $\check{\eta}$.
Similarly, first choosing $h$ and then choosing $\check{\del}$ also sufficiently small depending on $\check{\eta}$, we find that
\begin{equation*}
% \label{eqn: intial poly approximation of tilde u star}
\bigg{\|} \check{u}^\ast(\check{y}) - \frac{1}{2}|A\check{y}|^2\bigg{\|}_{L^\infty(A^{-1}B_{(16h)^{1/2}} \cap \{\check{y}_n \geq \check{\zeta}(\check{y}_0) \})}  \leq \frac{1}{2}\check{\eta}h.
\end{equation*}

Now, notice that the sections of $\check{u}$ and $\check{v}$ appropriately restricted are comparable if $\check{\del}$ is sufficiently small.
In particular, choose $h \leq 2\check{h}/3$ and take $\check{\del}$ small enough so that $(C_2 + 1)\om(\check{\del})^{1/2} \leq h/2$.
Then, recalling \eqref{eqn: sections of tilde v}, we have that
\begin{equation}
\label{eqn: sections are comparable}
Z_{\check{v}}(h/2) \cap \{\check{x}_n \geq \check{\gamma}(\check{x}_0) \} \subset S_{\check{u}}(h) \subset Z_{\check{v}}(3h/2) \cap \{\check{x}_n \geq \check{\gamma}(\check{x}_0) \} \Subset B_{1/9}
\end{equation}
where
\[
S_{\check{u}}(h) := \{ \check{x} \in \check{\C} : \check{u}(\check{x}) \leq h \}.
\]

Furthermore, using \eqref{eqn: C3 estimate}, we deduce that
\[
\mathcal{E}(h) \subset Z_{\check{v}}(h+ C_1(2C_1h)^{3/2})
\qquad\text{and}\qquad
Z_{\check{v}}(h) \subset \mathcal{E}(h+ C_1(2C_1h)^{3/2})
\]
where
\[
\mathcal{E}(h) := \bigg{\{} \check{x} : \frac{1}{2}D^2\check{v}(0)\check{x}\cdot\check{x} < h \bigg{\}} = AB_{(2h)^{1/2}}.
\]
Hence, \eqref{eqn: sections are comparable} implies that
\begin{equation}
\label{eqn: sublevel of u tilde comparable}
AB_{(h/7)^{1/2}} \cap \{\check{x}_n \geq \check{\gamma}(\check{x}_0) \} \subset S_{\check{u}}(h) \subset AB_{(7h)^{1/2}} \cap \{\check{x}_n \geq \check{\gamma}(\check{x}_0) \}
\end{equation}
if $\check{\delta}$ and $h$ are small enough.
In addition, arguing as in the proof of \cite[Theorem~4.3]{DF}, we see that
\begin{equation}
\label{eqn: image of sublevel of u tilde comparable}
A^{-1}B_{(h/7)^{1/2}} \cap \{\check{y}_n \geq \check{\zeta}(\check{y}_0) \} \subset \pa_\ast \check{u}(S_{\check{u}}(h)) \subset A^{-1}B_{(7h)^{1/2}} \cap \{\check{y}_n \geq \check{\zeta}(\check{y}_0) \},
\end{equation}
provided that $\check{\del}$ and $h$ are sufficiently small.
\\

{\it -- Case 1: $x^0 \in \{ x_n = \gamma(x_0) \}$, i.e., $\check{\gamma}(0) = 0$.}
\\

{\it -- Step 4.1: A change of variables.}\\

Recalling that $A$ is symmetric, $\det(A) = 1$, $C_1^{-1/2}\Id \leq A \leq C_1^{1/2}$, \eqref{eqn: symmetry Hessian at 0}, and \eqref{eqn: Hessian of tilde v at 0}, a simple computation shows that there exists a matrix $M$ such that
\[
\det(M) = 1,
\]
the matrix $M^{-1}A^{-1}$ is symmetric,
\[
(M^{-1}A^{-1})_{ni} = (M^{-1}A^{-1})_{in} = 0 \qquad \forall i = 1, \dots, n-1,
\]
and
\[
|M - \Id| \leq C_3\check{\del}.
\]
(In this case, the $\check{\del}$ factor comes from the H\"{o}lder semi-norm of the gradient of $\check{\gamma}$ only.)
Now, consider the change of variables
\[
\tilde{x} = \frac{1}{h^{1/2}}M^{-1}A^{-1} \check{x} \qquad\text{and}\qquad
\tilde{y} := \frac{1}{h^{1/2}}M^\ast A \check{y}.
\]
Let
\[
u_1(\tilde{x}) := \frac{1}{h}\check{u}(h^{1/2}A M \tilde{x})
\]
and set
\[
\C_1 := S_{u_1}(1) \qquad\text{and}\qquad \K_1 := \pa_\ast u_1(S_{u_1}(1)).
\]
As $\det(M)=\det(A)= 1$, we deduce that 
\[
(\nab u_1)_{\#}f_1 = g_1
\]
with
\[
f_1(\tilde{x}) := \check{f}(h^{1/2}AM\tilde{x})\1_{\C_1} \qquad\text{and}\qquad g_1(\tilde{y}) := \check{g}(h^{1/2}A^{-1}(M^\ast)^{-1}\tilde{y})\1_{\K_1}.
\]
Then, from \eqref{eqn: sublevel of u tilde comparable}, \eqref{eqn: image of sublevel of u tilde comparable}, and our estimate on $M$,
\[
B_{1/3} \cap \{ \tilde{x}_n \geq \gamma_1(\tilde{x}_0) \}  \subset \C_1 \subset B_3 \cap \{ \tilde{x}_n \geq \gamma_1(\tilde{x}_0) \}
\]
and
\[
B_{1/3} \cap \{ \tilde{y}_n \geq \zeta_1(\tilde{y}_0) \} \subset \K_1 \subset B_3 \cap \{ \tilde{y}_n \geq \zeta_1(\tilde{y}_0) \}.
\]
Additionally,
\[
\gamma_1(0) = 0, \qquad \frac{1}{C_4h^{1/2}}|\check{\zeta}(0)| \leq |\zeta_1(0)| \leq \frac{C_4}{h^{1/2}}|\check{\zeta}(0)|, \qquad \nab \gamma_1(0) = \nab\zeta_1(0) = 0,
\]
and
\[
[\nab \gamma_1]_{C^{0,\alpha}(\B_3)} + [\nab \zeta_1]_{C^{0,\alpha}(\B_3)} \leq c_0\check{\delta} =: \del_1
\]
for $c_0 < 1$, choosing $h$ and $\check{\delta}$ small enough so that $C_4h^{\alpha/2} \ll 1$.
\\

{\it -- Step 5.1: An iteration scheme.}\\

In order run Steps 2 through 4.1 on $u_1$, $\C_1$, $\K_1$, $\gamma_1$, $\zeta_1$, $f_1$, and $g_1$, we need to ensure two things: 1. that the hypotheses of Lemma~\ref{lem: cptness} and 2. that we can ensure the regularity of the convex potential $v_1$ we would produce after applying Lemma~\ref{lem: cptness}.

So long as $|\zeta_1(0)| \leq 1/27$, 1. is satisfied.

Now, let us move to understanding point 2.
By \eqref{eqn: intial poly approximation of tilde u}, 
\[
\bigg{\|} u_1(\tilde{x}) - \frac{1}{2}|M\tilde{x}|^2\bigg{\|}_{L^\infty(M^{-1}B_4 \cap \{ \tilde{x}_n \geq \gamma_1(\tilde{x}_0) \})} \leq \frac{1}{2}\check{\eta},
\]
from which, using our estimate on $M$, we find that
\begin{equation}
\label{eqn: diagonalized u close to parabola bdry}
\bigg{\|} u_1(\tilde{x}) - \frac{1}{2}|\tilde{x}|^2\bigg{\|}_{L^\infty(B_3 \cap \{ \tilde{x}_n \geq \gamma_1(\tilde{x}_0) \})} \leq \frac{1}{2}\check{\eta} + C_5\check{\delta} \leq \check{\eta}.
\end{equation}
Applying Lemma~\ref{lem: cptness} and arguing as before, \eqref{eqn: convexity for reg} becomes
\begin{equation}
\label{eqn: approx in iteration}
\bigg{\|}v_1(\tilde{x})-\frac{1}{2}|\tilde{x}|^2\bigg{\|}_{L^\infty(B_{1/6})} 
\leq C_0(\check{\eta} + \om(\check{\del}) + |\zeta_1(0)|).
\end{equation}

Let $\rho \leq 1/27$ is the largest replacement for $|\zeta_1(0)|$ in \eqref{eqn: approx in iteration} that permits \eqref{eqn: MA equation for v tilde} with $v_1$ replacing $\tilde{v}'$.
In order to continue with Step 2, we need that
\[
|\zeta_1(0)| \leq \rho.
\]
Notice that if we decrease $\check{\eta}$ (and then necessarily $\check{\delta}$), we can increase $\rho$.
In particular, we can ensure that 
\begin{equation}
\label{eqn: rho large}
\rho \gg \check{\eta}^{1/2} + \check{\delta}.
\end{equation}
So, provided that $\check{\eta},\check{\del}$, and $r$ are sufficiently small to guarantee 1., we can indeed continue and find $A_1$ and $M_1$ and define $u_2$, $\C_2$, $\K_2$, $\gamma_2$, $\zeta_2$, $f_2$, and $g_2$.
Notice that the only differences between this family and its predecessor is that $|\zeta_2(0)|$ will increase:
\[
\frac{1}{C_4h^{1/2}}|\zeta_1(0)| \leq |\zeta_2(0)| \leq \frac{C_4}{h^{1/2}}|\zeta_1(0)|
\]
and the H\"{o}lder semi-norm of the lower boundaries of $\C_2$ and $\K_2$ will decrease:
\[
[\nab \gamma_2]_{C^{0,\alpha}(\B_3)} + [\nab \zeta_2]_{C^{0,\alpha}(\B_3)} \leq c_0\del_1 = c_0^2\check{\delta}.
\]
Hence, if $|\zeta_2(0)| \leq \rho$, then we can repeat our procedure again and iterate further.

Suppose $\check{\zeta}(0) \neq 0$, i.e., $y^0 \notin \{ y_n = \zeta(y_0) \}$.
Then, there will be a first time $k \geq 2$ at which
\[
|\zeta_{k}(0)| > \rho,
\]
and we can no longer continue our iterative procedure.
That said, recalling how we proved \eqref{eqn: diagonalized u close to parabola bdry}, we have that
\[
\bigg{\|} u_k(\tilde{x}) - \frac{1}{2}|\tilde{x}|^2\bigg{\|}_{L^\infty(B_3 \cap \{ \tilde{x}_n \geq \gamma_k(\tilde{x}_0)\})} \leq \check{\eta}.
\]
Consequently,
\[
\pa u_k(\tilde{x}) \subset \bigg{\{} \tilde{y}_n \geq - 4\check{\delta} - \frac{5}{2}\check{\eta}^{1/2} \bigg{\}}
\]
for all $\tilde{x} \in \{\tilde{x}_n \geq \gamma_k(\tilde{x}_0) + \check{\eta}^{1/2}\} \cap \C_k$ (cf. \eqref{eqn: bar y - bar x estimate}).
On the other hand,
\[
|\C_k \setminus \{ \tilde{x}_n \geq \gamma_k(\tilde{x}_0) + \check{\eta}^{1/2} \}| \leq C\check{\eta}^{1/2}
\]
and, recalling \eqref{eqn: rho large},
\[
\bigg{|}\K_k \setminus \bigg{\{} y_n \geq - 4\check{\delta} - \frac{5}{2}\check{\eta}^{1/2} \bigg{\}}\bigg{|} \geq c\bigg{(}\rho - 8\check{\delta} - \frac{5}{2}\check{\eta}^{1/2}\bigg{)} > 0.
\]
But these two inequalities together violate the transport condition $(\nab u_k)_{\#}f_k = g_k$ if $\check{\delta}$ and $\check{\eta}$ are sufficiently small (again, recall \eqref{eqn: rho large}).

As $\check{\zeta}(0) = 0$ (that is, $y^0 \in \{ y_n = \zeta(y_0) \}$), we can iterate indefinitely.
In turn, for all $k \geq 1$, we have determinant one matrices $A_k$ and $M_k$ such that
\[
\frac{1}{C_4} \Id \leq A_kM_k \leq C_4 \Id
\]
and
\[
D_kB_{h^{k/2}/7^{1/2}} \cap \{ \tilde{x}_n \geq \gamma_1(\tilde{x}_0) \} \subset S_{u_1}(h^k) \subset D_kB_{7^{1/2}h^{k/2}} \cap \{ \tilde{x}_n \geq \gamma_1(\tilde{x}_0) \} 
\]
where $D_k := A_1M_1 \cdots A_{k-1}M_{k-1}A_k$.
Hence,
\[
B_{(h^{1/2}/3C_4)^k} \cap \{ \tilde{x}_n \geq \gamma_1(\tilde{x}_0) \} \subset S_{u_1}(h^k) \subset B_{(3C_4h^{1/2})^k} \cap \{ \tilde{x}_n \geq \gamma_1(\tilde{x}_0) \} \qquad \forall k \geq 1.
\]
Thus, fixing $\beta \in (0,1)$ and then choosing $h$ sufficiently small and $d := h^{1/2}/3C_4$, it follows that
\[
\|u_1\|_{L^\infty(B_{d^k} \cap \{ \tilde{x}_n \geq \gamma_1(\tilde{x}_0)\})} \leq d^{(1+\beta)k}.
\]
Since $x^0$ was an arbitrary point on $B_r \cap \{ x_n = \gamma(x_0)\}$, we have that $u \in C^{1,\beta}(B_r \cap \{ x_n = \gamma(x_0)\})$.
\\

{\it -- Case 2: $x^0 \in \{ x_n > \gamma(x_0)\}$.}
\\

{\it -- Step 4.2: A change of variables.}\\

Notice that in Case 1, we showed that
\[
\pa u(B_r \cap \{ x_n = \gamma(x_0) \}) = \nab u(B_r \cap \{ x_n = \gamma(x_0) \}) \subset B_{1/2} \cap \{ y_n = \zeta(y_0) \}.
\]
By duality, that is, considering inverse transport $\nab u^\ast$, it follows that 
\begin{equation}
\label{eqn: int to int}
\pa u(B_r \cap \{ x_n > \gamma(x_0) \}) \subset B_{1/2} \cap \{ y_n > \zeta(y_0) \}.
\end{equation}

Just as before, we find a determinant one matrix $M$ such that the matrix $M^{-1}A^{-1}$ is symmetric and has eigenvectors $\{e_1,\dots,e_{n-1},\e_n\}$ for which $\e_n \cdot e_i = 0$ for all $i = 1, \dots, n-1$.
However, now that $\check{\gamma}(0) \neq 0$, we find that
\begin{equation}
\label{eqn: M is close to Id}
|M - \Id| \leq C_3(|\check{\gamma}(0)| + \check{\delta}) < C_3(|\check{\gamma}(0)| + 2\check{\delta}) < C_3(\rho + \check{\delta}).
\end{equation}
If we define $\gamma_1$ and $\zeta_1$ as we did in Step 4.1, then 
\[
|\nab \zeta_1(0)| \leq C_6|\nab \check{\zeta}(0)|.
\]

There are two subcases two consider: 1. $C_6 \leq 1$ and 2. $C_6 > 1$.

In Subcase 1, we consider the same change of variables as we did in Step 4.1.
Then, from \eqref{eqn: sublevel of u tilde comparable}, \eqref{eqn: image of sublevel of u tilde comparable}, and construction,
\[
B_{1/3} \cap \{ \tilde{x}_n \geq \gamma_1(\tilde{x}_0) \}  \subset \C_1 \subset B_3 \cap \{ \tilde{x}_n \geq \gamma_1(\tilde{x}_0) \}
\]
and
\[
B_{1/3} \cap \{ \tilde{y}_n \geq \zeta_1(\tilde{y}_0) \} \subset \K_1 \subset B_3 \cap \{ \tilde{y}_n \geq \zeta_1(\tilde{y}_0) \}
\]
provided the right-hand side of \eqref{eqn: M is close to Id} is sufficiently small.
Additionally,
\[
\frac{1}{C_7h^{1/2}}|\check{\gamma}(0)| \leq |\gamma_1(0)| \leq \frac{C_7}{h^{1/2}}|\check{\gamma}(0)|, \qquad \frac{1}{C_7h^{1/2}}|\check{\zeta}(0)| \leq |\zeta_1(0)| \leq \frac{C_7}{h^{1/2}}|\check{\zeta}(0)|, 
\]
\[
\nab \gamma_1(0) = 0, \qquad |\nab \zeta_1(0)| \leq \check{\delta},
\]
and
\begin{equation}
\label{eqn: seminorm dec subcase 1}
[\nab \gamma_1]_{C^{0,\alpha}(\B_3)} + [\nab \zeta_1]_{C^{0,\alpha}(\B_3)} \leq c_1\check{\delta}
\end{equation}
for $c_1 \ll 1$, taking $h$ smaller.

In Subcase 2, we additionally apply a shearing transformation $L^\ast$ to swap which side has a horizontal tangent at $0$ for the function defining the lower boundary.
More precisely, there exists a shearing transformation $L$ so that 
\[
\nab \zeta_1(0) = 0 \qquad\text{and}\qquad |L - \Id| \leq C_6\check{\del},
\]
defining
\[
\tilde{y} := \frac{1}{h^{1/2}}L^\ast M^\ast A\check{y}
\]
and letting $\zeta_1 : \B_3 \to \R$ be such that
\[
\{ \tilde{y}_n \geq \zeta_1(\tilde{y}_0) \} = \frac{1}{h^{1/2}}L^\ast M^\ast A\{\check{y}_n \geq \check{\zeta}(\check{y}_0) \}.
\]
So, considering the change of variables
\[
\tilde{x} := \frac{1}{h^{1/2}}L^{-1}M^{-1}A^{-1} \check{x} \qquad\text{and}\qquad
\tilde{y} := \frac{1}{h^{1/2}}L^\ast M^\ast A \check{y},
\]
we define
\[
u_1(\tilde{x}) := \frac{1}{h}\check{u}(h^{1/2}A M L\tilde{x})
\]
and set
\[
\C_1 := S_{u_1}(1) \qquad\text{and}\qquad \K_1 := \pa_\ast u_1(S_{u_1}(1)).
\]
As $\det(L) = \det(M)=\det(A)= 1$, we deduce that 
\[
(\nab u_1)_{\#}f_1 = g_1
\]
with
\[
f_1(\tilde{x}) := \check{f}(h^{1/2}AML\tilde{x})\1_{\C_1} \qquad\text{and}\qquad g_1(\tilde{y}) := \check{g}(h^{1/2}A^{-1}(M^\ast)^{-1}(L^\ast)^{-1}\tilde{y})\1_{\K_1}.
\]

From \eqref{eqn: sublevel of u tilde comparable}, \eqref{eqn: image of sublevel of u tilde comparable}, and our estimates on $M$ and $L$,
\[
B_{1/3} \cap \{ \tilde{x}_n \geq \gamma_1(\tilde{x}_0) \}  \subset \C_1 \subset B_3 \cap \{ \tilde{x}_n \geq \gamma_1(\tilde{x}_0) \}
\]
and
\[
B_{1/3} \cap \{ \tilde{y}_n \geq \zeta_1(\tilde{y}_0) \} \subset \K_1 \subset B_3 \cap \{ \tilde{y}_n \geq \zeta_1(\tilde{y}_0) \}.
\]
Additionally,
\[
\frac{1}{C_8h^{1/2}}|\check{\gamma}(0)| \leq |\gamma_1(0)| \leq \frac{C_8}{h^{1/2}}|\check{\gamma}(0)|, \qquad \frac{1}{C_8h^{1/2}}|\check{\zeta}(0)| \leq |\zeta_1(0)| \leq \frac{C_8}{h^{1/2}}|\check{\zeta}(0)|, 
\]
\[
|\nab \gamma_1(0)| \leq \check{\delta}, \qquad \nab \zeta_1(0) = 0,
\]
and
\[
[\nab \gamma_1]_{C^{0,\alpha}(\B_3)} + [\nab \zeta_1]_{C^{0,\alpha}(\B_3)} \leq c_1\check{\delta}
\]
where for the inequality $|\nab \gamma_1(0)| \leq \check{\delta}$, we have used \eqref{eqn: seminorm dec subcase 1}.
\\

{\it -- Step 5.2: An iteration scheme.}\\

In order run Steps 2 through 4.2 on $u_1$, $\C_1$, $\K_1$, $\gamma_1$, $\zeta_1$, $f_1$, and $g_1$, like before, we need to ensure two things: 1. that the hypotheses of Lemma~\ref{lem: cptness} and 2. that we can ensure the regularity of the convex potential $v_1$ we would produce after applying Lemma~\ref{lem: cptness}.

So long as $|\gamma_1(0)|, |\zeta_1(0)| \leq 1/27$, 1. is satisfied.

Now, let us move to understanding point 2.
By \eqref{eqn: intial poly approximation of tilde u}, 
\[
\bigg{\|} u_1(\tilde{x}) - \frac{1}{2}|ML\tilde{x}|^2\bigg{\|}_{L^\infty(L^{-1}M^{-1}B_4 \cap \{ \tilde{x}_n \geq \gamma_1(\tilde{x}_0) \})} \leq \frac{1}{2}\check{\eta},
\]
from which, using our estimates on $M$ and $L$, we find that
\[
\bigg{\|} u_1(\tilde{x}) - \frac{1}{2}|\tilde{x}|^2\bigg{\|}_{L^\infty(B_3 \cap \{ \tilde{x}_n \geq \gamma_1(\tilde{x}_0) \})} \leq \check{\eta} + C_9|\check{\gamma}(0)|.
\]
Applying Lemma~\ref{lem: cptness} and arguing as before, \eqref{eqn: convexity for reg} becomes
\[
\begin{split}
\bigg{\|}v_1(\tilde{x})-\frac{1}{2}|\tilde{x}|^2\bigg{\|}_{L^\infty(B_{1/6})} 
&\leq C_0(\check{\eta} + \om(\check{\del}) + C_9|\check{\gamma}(0)| + |\gamma_1(0)|+ |\zeta_1(0)| + \check{\del})\\
&\leq \tilde{C}_0(\check{\eta} + \om(\check{\del}) + |\check{\gamma}(0)| + |\gamma_1(0)|+ |\zeta_1(0)|).
\end{split}
\]
In this case, we need
\[
|\check{\gamma}(0)| + |\gamma_1(0)|+ |\zeta_1(0)| \leq \rho
\]
to proceed, decreasing $\rho$ to account for the larger factor $\tilde{C}_0$.
Recall that  
\begin{equation}
\label{eqn: rho large 2}
\rho \gg \check{\eta}^{1/2} + \check{\del}.
\end{equation}
Hence, if $\check{\eta}, \check{\del}$, and $r$ are sufficiently small, we can indeed continue and find $A_1$ a symmetric, determinant one matrix such that
\[
\frac{1}{C_1^{1/2}} \Id \leq A_1 \leq C_1^{1/2} \Id,
\]
\[
A_1B_{(h/7)^{1/2}} \cap \{\tilde{x}_n \geq \gamma_1(\tilde{x}_0) \} \subset S_{u_1}(h) \subset A_1B_{(7h)^{1/2}} \cap \{\tilde{x}_n \geq \gamma_1(\tilde{x}_0) \},
\]
\[
A_1^{-1}B_{(h/7)^{1/2}} \cap \{\tilde{y}_n \geq \zeta_1(\tilde{y}_1) \} \subset \pa_\ast u_1(S_{u_1}(h)) \subset A_1^{-1}B_{(7h)^{1/2}} \cap \{\tilde{y}_n \geq \zeta_1(\tilde{y}_0) \},
\]
and
\[
\bigg{\|} u_1(\tilde{x}) - \frac{1}{2}|A_1^{-1}\tilde{x}|^2\bigg{\|}_{L^\infty(A_1B_{(16h)^{1/2}} \cap \{\tilde{x}_n \geq \gamma_1(\tilde{x}_0) \})}  + \bigg{\|} u_1^\ast(\tilde{y}) - \frac{1}{2}|A_1\tilde{y}|^2\bigg{\|}_{L^\infty(A_1^{-1}B_{(16h)^{1/2}} \cap \{\tilde{y}_n \geq \zeta_1(\tilde{y}_0) \})}
\leq \check{\eta}h.
\]
Using the same construction as before, we build $M_1$ and if needed $L_1$ (this time however $L_1$ will make the tangent plane at $0$ to lower boundary of the source horizontal and the tangent plane at $0$ to the lower boundary of the target smaller than $\check{\del}$) and define $u_2$, $\C_2$, $\K_2$, $\gamma_2$, $\zeta_2$, $f_2$, and $g_2$.
Now,
\[
|M_1 - \Id| \leq C_3(|\gamma_1(0)| + 2\check{\delta}) < C_3(\rho + \check{\delta}),
\]
\[
[\nab \gamma_2]_{C^{0,\alpha}(\B_3)} + [\nab \zeta_2]_{C^{0,\alpha}(\B_3)} \leq c_1^2\check{\delta},
\]
\[
\frac{1}{C_8h^{1/2}}|\gamma_1(0)| \leq |\gamma_2(0)| \leq \frac{C_8}{h^{1/2}}|\gamma_1(0)|, \qquad\text{and}\qquad \frac{1}{C_8h^{1/2}}|\zeta_1(0)| \leq |\zeta_2(0)| \leq \frac{C_8}{h^{1/2}}|\zeta_1(0)|.
\]
Continuing, there will be a first time $k \geq  2$ when
\[
2(|\zeta_k(0)| + |\gamma_k(0)|) \geq |\gamma_{k-1}(0)| + |\gamma_k(0)| + |\zeta_k(0)| > \rho.
\]
At this point, we go back to $u_{k-1}$ and consider 
\[
\tilde{u}(\tilde{x}) := \frac{1}{h}u_{k-1}(h^{1/2}A_{k-1}\tilde{x})
\]
(and, correspondingly, $\tilde{\C}$, $\tilde{\K}$, $\tilde{\gamma}$, and $\tilde{\zeta}$) rather than $u_{k}$, forgetting about $M_{k-1}$ and $L_{k-1}$.
Notice that
\[
B_{1/3} \cap \{ \tilde{x}_n \geq \tilde{\gamma}(\tilde{x}_0) \}  \subset \C_1 \subset B_3 \cap \{ \tilde{x}_n \geq \tilde{\gamma}(\tilde{x}_0) \},
\]
\[
B_{1/3} \cap \{ \tilde{y}_n \geq \tilde{\zeta}(\tilde{y}_0) \} \subset \K_1 \subset B_3 \cap \{ \tilde{y}_n \geq \tilde{\zeta}(\tilde{y}_0) \},
\]
and
\begin{equation}
\label{eqn: final close to parabola}
\bigg{\|} \tilde{u}(\tilde{x}) - \frac{1}{2}|\tilde{x}|^2\bigg{\|}_{L^\infty(B_3 \cap \{\tilde{x}_n \geq \tilde{\gamma}(\tilde{x}_0) \})} +  \bigg{\|} \tilde{u}^\ast(\tilde{y}) - \frac{1}{2}|\tilde{y}|^2\bigg{\|}_{L^\infty(B_3 \cap \{\tilde{y}_n \geq \tilde{\zeta}(\tilde{y}_0) \})}
\leq \check{\eta}.
\end{equation}
Moreover,
\[
C(|\tilde{\gamma}(0)| + |\tilde{\zeta}(0)|) \geq |\gamma_{k}(0)|+|\zeta_{k}(0)|.
\]
Hence, using \eqref{eqn: final close to parabola}, arguing as we did to prove \eqref{eqn: image of point under subdiff}, \eqref{eqn: int to int}, and recalling that \eqref{eqn: rho large 2}, we deduce that
\[
B_{c\rho} \subset \tilde{\C} \qquad\text{and}\qquad B_{c\rho} \subset \tilde{\K}.
\]
So, we are in an interior situation, and taking $\check{\eta}$ (and also $\check{\del}$) sufficiently small depending on $\rho$, we can apply the arguments of \cite[Theorem~4.3]{DF} to conclude that $u \in C^{1,\beta}(x^0)$.
As $x^0 \in B_r \cap \{ x_n > \gamma(x_0) \}$ was arbitrary, the theorem holds.
\end{proof}

%~~~~~~~~~~~~~~~~~~~~~~~~~~~~~~~~~~~~~~~~~~~~~~~~~~~~~~~~~%
\section{Perturbations in Non-regular Domains}

In the previous section, we considered perturbations in regions of domains that are at least $C^{1,\alpha}$, and, in the case of Theorem~\ref{thm: c1b at bdry}, we additionally considered non-constant densities.
The next natural question is, what can be said about domain and density perturbations of less regular portions of a domain?
We have seen, in some sense, that corners destroy regularity, and so this question is rather delicate.
Hence, we consider the simple situation of rectangles in $\R^2$, wherein one might hope to leverage the highly symmetric nature of these domains to say something. 

Interestingly, Theorem~\ref{thm: c1b at bdry} can be extended to domains in two dimensions that are deformation of domains with 90 degree corners.
The fundamental domain in Theorem~\ref{thm: c1b at bdry} is an upper half ball; in the sense given in its hypotheses, $\C$ and $\K$ are comparable to upper half balls.
If $\gamma, \zeta \equiv 0$, then the points on $\{ x_2 = \gamma(x_1) \}$ are interior points for the optimal transport for the data reflected over horizontal lines, and regularity follows from \cite[Proposition~2]{FK}.
In Theorem~\ref{thm: c1b at bdry}, the regularity of $u$ at/near the lower boundary is inherited from the interior regularity of a potential to an approximating problem that takes advantage of this ``reflection symmetry yields regularity'' argument.
The same strategy will extend Theorem~\ref{thm: c1b at bdry} when considering an upper quarter ball as a fundamental domain.
Just as points on flat boundaries can be turned into interior points, 90 degree corners and points near these corners can be turn into interior points.
(See, e.g., Lemma~\ref{lem: C1a up to bdry Q}.)
We leave the details of Theorem~\ref{thm: c1b at bdry}'s extension to the interested reader.

Now let us move to considering higher order density perturbations in corners: given two densities $f, g \in C^{\infty}(\overline{Q})$ bounded away from zero and satisfying the mass balance condition $\|f\|_{L^1(Q)} = \|g\|_{L^1(Q)}$ where $Q := (0,1) \times (0,1)$, is the optimal transport $T$ taking $f$ to $g$ of class $C^{\infty}(\overline{Q})$?

Set
\[
\Up_{\rm b} := (0,1) \times \{0\},\quad \Up_{\rm t} :=  (0,1) \times \{1\},\quad \Up_{\rm l} := \{0\} \times (0,1), \quad\text{and}\quad \Up_{\rm r} := \{1\} \times (0,1).
\]
In what follows, we let $C$ be a generic positive constant; it may change from line to line, and its dependences, if any, will either be clear from context or explicitly given.

First, notice that $T \in C^{\infty}_{\rm loc}(Q) \cap C^{0,\sigma}(\overline{Q})$ for some $\sigma < 1$ (see \cite{C1,C2}).
So, to start, we show, taking advantage of the reflection symmetries of $Q$, that the non-uniform convexity of $Q$ does not prohibit $C^{2,\alpha}$-regularity given $\alpha$-H\"{o}lder continuous densities on $\overline{Q}$ bounded away from zero.
In other words, the symmetries of $Q$ allow us to recover the same regularity up to the boundary of our transport as we would had our source and target domains been uniformly convex. 
(See \cite{C3}.)

\begin{lemma}
\label{lem: C1a up to bdry Q}
Let $f, g \in C^{0,\a}(\overline{Q})$ be bounded away from zero and satisfy the mass balance condition $\|f\|_{L^1(Q)} = \|g\|_{L^1(Q)}$.
The optimal transport $T$ is a diffeomorphism of class $C^{1,\a}(\overline{Q})$.
Moreover, $T$ maps each segment of the boundary of $Q$ diffeomorphically to itself.
\end{lemma}

\begin{proof}
First, by \cite{C2}, $T : \overline{Q} \to \overline{Q}$ is a bi-H\"{o}lder continuous homeomorphism.
Now, set 
\[
Q'' := (-1,1) \times (-1,1)
\] 
and let $f''$ be the even reflection of $f$ around the origin to $Q''$:
\[
f''(x) :=
\begin{cases}
f(x_1,x_2) &\text{in } [0,1] \times [0,1]\\
f(x_1,-x_2) &\text{in } [0,1]\times(0,-1]\\
f(-x_1,x_2) &\text{in } [-1,0)\times[0,1]\\
f(-x_1,-x_2) &\text{in } [-1,0)\times(0,-1].
\end{cases}
\]
Also, let $g''$ be the even reflection of $g$ around the origin to $Q''$.
By construction, $f''$ and $g''$ are of class $C^{0,\a}(Q'')$; and so, $T'' \in C^{1,\a}(B_{3/4})$, by \cite{C1}, where $T''$ is the optimal transport taking $f''$ to $g''$. 
By symmetry, $T''([-1,1] \times \{0\}) = [-1,1] \times \{0\}$, $T''(\{0\} \times [-1,1]) = \{0\} \times [-1,1]$, and the restriction of $T''$ to $\overline{Q}$ is $T$, the optimal transport taking $f$ to $g$.
It follows that $T \in C^{1,\a}(B_{3/4} \cap Q)$.
The lemma then follows after similarly reflecting $f$ and $g$ evenly around the points $(1,0), (0,1)$, and $(1,1)$.
\end{proof}

\begin{remark}
We can also see that $T = \nab u$ maps each segment of the boundary of $Q$ to itself and fixes the corners of $Q$ via a local argument. 
If, say, $\nab u(x) \in \Up_{\rm r}$ for some $x \in \Up_{\rm b}$, then $\nab u(\ell_x) \subset \overline{\Up}_{\rm r}$ for some $\ell_x$ non-empty subsegment of $\Up_{\rm b}$ containing $x$ (possibly as an endpoint).
Thus, $\pa_1 u(\cdot,0)$ is constant on $\ell_x$, or, equivalently,
$u|_{\ell_x}$ is linear.
However, this contradicts the strict convexity of $u$ along the boundary of $Q$ given by \cite{C2}.
Finally, $\nab u(x) \notin \Up_{\rm t}$ for any $x \in \Up_{\rm b}$ by the monotonicity of $\nab u$.
% ; indeed, $\nu(x) \cdot \nu(\nab u(x)) \geq 0$, where $\nu$ denotes the outer unit normal to $Q$.
\end{remark}

As we are working in two dimensions, rather than consider a Monge-Amp\`{e}re equation, we can instead consider a quasi-linear, uniformly elliptic equation for the partial Legendre transform $u^\star$ of $u$; and after absorbing the coefficients' dependences on $u^\star$ at the expense of their regularity, we can consider a linear, uniformly elliptic equation.
This observation will play a key role in answering our question.

\begin{theorem}
\label{thm: C2a up to bdry Q}
Let $f, g \in C^{1,\a}(\overline{Q})$ be bounded away from zero and satisfy the mass balance condition $\|f\|_{L^1(Q)} = \|g\|_{L^1(Q)}$.
The optimal transport $T$ is a diffeomorphism of class $C^{2,\a}(\overline{Q})$.
Moreover, $T$ maps each segment of the boundary of $Q$ diffeomorphically to itself.
\end{theorem}

\begin{proof}
By Lemma~\ref{lem: C1a up to bdry Q}, given any $f$ and $g$ of class $C^{1,\a}(\overline{Q})$ bounded away from zero, any convex potential $u$ of the optimal transport $T$ taking $f$ to $g$ is $C^{2,\a}(\overline{Q})$.
Moreover, because $T$ maps each segment of the boundary of $Q$ to itself, we see that $\pa_\nu u = 0$ on $\Up_{\rm b} \cup \Up_{\rm l}$ and $\pa_\nu u = 1$ on $\Up_{\rm t} \cup \Up_{\rm r}$.
In particular, $\nab u(0) = 0$.
Hence,
\begin{equation}
\label{eqn: neumann MA}
\begin{cases}
\det D^2 u = f/g(\nab u) &\text{in }Q\\
\pa_\nu u = 0 &\text{on } \Up_{\rm b} \cup \Up_{\rm l}\\
\pa_\nu u = 1 &\text{on } \Up_{\rm t} \cup \Up_{\rm r}.
\end{cases}
\end{equation}
Furthermore, $u$ is uniformly convex as $f/g > 0$.
In other words,
\begin{equation}
\label{eqn: strong convexity}
0 < \frac{1}{C} \leq \pa_{ii}u(x) \leq C \qquad\forall x \in \overline{Q}.
\end{equation}

Now, let $u^\star$ be the partial Legendre transform of $u$ in the $\e_1$-direction:
\[
u^\star(p,x_2) := \sup_{x_1 \in \overline{Q}_{x_2}} \{ px_1 - u(x_1,x_2)\}.
\]
Here, $\overline{Q}_{x_2}$ is the horizontal slice of $\overline{Q}$ at height $x_2$.
Notice that the point $x_1 = X_1(p,x_2)$ where this supremum is attained is characterized by the equation
\begin{equation}
\label{eqn: plt eqn}
\pa_1 u(X_1(p,x_2),x_2) = p.
\end{equation}
Since $u(\cdot,x_2)$ is strictly convex and of class $C^1(\overline{Q}_{x_2})$, we have that $\pa_1 u(\cdot,x_2)$ is injective and \eqref{eqn: plt eqn} is uniquely solvable given a pair $(p,x_2) \in \overline{Q}$.
The map $(x_1,x_2) \mapsto (\pa_1u(x_1,x_2),x_2)$ takes $\overline{Q}$ to $\overline{Q}$ as $T = \nab u$ maps $\overline{Q}$ to $\overline{Q}$.
Recall that the first partial derivatives of $u^\star$ are related to the first partial derivatives of $u$ by the equations
\[
\pa_1 u^\star(p,x_2) = X_1(p,x_2) \qquad\text{and}\qquad \pa_2 u^\star(p,x_2) = -\pa_2 u(X_1(p,x_2),x_2),
\]
while the pure second partial derivatives of $u^\star$ are related to the pure second partial derivatives of $u$ by the equations
\[
\pa_{11}u^\star(p,x_2) = \frac{1}{\pa_{11}u(X_1(p,x_2),x_2)} \qquad\text{and}\qquad \pa_{22}u^\star(p,y) = \bigg[\frac{(\pa_{12}u)^2}{\pa_{11}u} - \pa_{22}u\bigg](X_1(p,x_2),x_2).
\]
Therefore, using \eqref{eqn: neumann MA}, \eqref{eqn: strong convexity}, and \eqref{eqn: plt eqn}, it follows that
\begin{equation}
\label{eqn: pLt of neumann MA}
\begin{cases}
f(\pa_1u^\star,x_2)\pa_{11}u^\star + g(p,-\pa_2 u^\star)\pa_{22}u^\star = 0 &\text{in }Q\\
\pa_\nu u^\star = 0 &\text{on } \Up_{\rm b} \cup \Up_{\rm l}\\
\pa_\nu u^\star = 1 &\text{on } \Up_{\rm t} \cup \Up_{\rm r}.
\end{cases}
\end{equation}
Applying Proposition~\ref{prop: schauder} to $v = u^\star$, $a_1 = a_1(p,x_2) = f(\pa_1u^\star(p,x_2),x_2)$, and $a_2 = a_2(p,x_2) = g(p,-\pa_2 u^\star(p,x_2))$, we see that $u^\star$ is of class $C^{3,\a}(\overline{Q}_{3/4})$.
By symmetry, we can treat each corner as the origin; it follows that $\|u^\star\|_{C^{3,\a}(\overline{Q})} \leq C$.

Finally, as $u$ is uniformly convex (\eqref{eqn: strong convexity}), $u$ and $u^\star$ have the same regularity on the closure of $Q$.
\end{proof}

Let $Q_r : = (0,r) \times (0,r) = rQ$.

\begin{proposition}
\label{prop: schauder}
Let $v \in C(\overline{Q})$ be such that $\|v\|_{L^\infty(\overline{Q})} \leq 1$ and
\[
\begin{cases}
\trace(AD^2 v) = 0 &\text{in } Q \\
\pa_\nu v = 0 &\text{on } \Up_{\rm b} \cup \Up_{\rm l}
\end{cases}
\]
where
\[
\l \Id \leq A := \diag(a_1,a_2) \leq \L \Id.
\]
If $\|A\|_{C^{1,\a}(\overline{Q})} \leq \L$, then 
\[
\|v\|_{C^{3,\a}(\overline{Q}_{3/4})} \leq C
\]
for some constant $C = C(\l,\L,\a) > 0$.
\end{proposition}

\begin{proof}
Up to a diagonal transformation, we can assume that $a_i(0) = 1$ for $i = 1,2$.
Then, up to zeroth order scaling, i.e., considering $v_r(x) := v(rx)$ and $A_r(x) := A(rx)$, we can assume that 
\[
\|A - \Id\|_{C^{1,\a}(\overline{Q})} \leq \vep
\]
for some $\vep > 0$ that will be chosen.

Let $P$ be a degree three polynomial such that
\[
\pa_1 P(0,\cdot) = \pa_2 P(\cdot,0) \equiv 0.
\]
Then, $P$ takes the form
\[
P(x) = p_0 + p_{2,1}x_1^2 + p_{2,2}x_2^2 + p_{3,1}x_1^3 + p_{3,2}x_2^3.
\]
Set
\[
\|P\| := \max_{m,j} |p_{m,j}|.
\]
Taylor expanding $A$ around the origin, we see that
\[
\trace(AD^2P) = \hat{P} + h 
\]
where $\hat{P}$ is a degree one polynomial with coefficients
\begin{equation}
\label{eqn: coeff}
\begin{cases}
d_0 \hspace{-0.25cm}&= 2p_{2,1} + 2p_{2,2}\\
d_{1,1} \hspace{-0.25cm}&= 6p_{3,1} + 2p_{2,1}\pa_1 a_1(0) + 2p_{2,2}\pa_1 a_2(0)\\
d_{1,2} \hspace{-0.25cm}&= 6p_{3,2} + 2p_{2,1}\pa_2 a_1(0) + 2p_{2,2}\pa_2 a_2(0)
\end{cases}
\end{equation}
and $h$ is such that
\begin{equation}
\label{eqn: conditions of remainder}
|h(x)| \leq  C\vep|x|^{1+\a}
\end{equation}
with $C = C(\|P\|) > 0$.
Notice that given any triplet $(d_0,d_{1,1},d_{1,2})$, the system \eqref{eqn: coeff} is uniquely solvable after choosing $p_{2,1}$.
Indeed,
\[
\begin{cases}
2p_{2,2} = d_0-2p_{2,1}\\
6p_{3,1} = d_{1,1} - (d_0-2p_{2,1})\pa_1 a_2(0) - 2p_{2,1}\pa_1 a_1(0)\\
6p_{3,2} = d_{1,2} - (d_0-2p_{2,1})\pa_2 a_2(0) - 2p_{2,1}\pa_2 a_1(0).
\end{cases}
\]

\begin{definition}
We call a degree three polynomial $P$ approximating for $v$ at zero if $\pa_1 P(0,\cdot) = \pa_2 P(\cdot,0) \equiv 0$ and $\hat{P} \equiv 0$.
\end{definition}

Recall $Q'' = (-1,1) \times (-1,1)$, and let $Q''_r := rQ''$.

\begin{lemma}
\label{lem: iteration}
Assume that for some $r \leq 1$ and some approximating polynomial $P$ for $v$ at zero with $\|P\| \leq 1$, we have that
\[
\|v-P\|_{L^\infty(\overline{Q}_r)} \leq r^{3+\a}.
\]
Then, there exists an approximating polynomial $\bar{P}$ for $v$ at zero such that
\[
\|v-\bar{P}\|_{L^\infty(\overline{Q}_{\rho r})} \leq (\rho r)^{3+\a}
\]
and
\[
\|P -\bar{P}\|_{L^\infty(Q''_r)} \leq Cr^{3+\a}
\]
for some constants $\rho = \rho(\l,\L,\a), C = C(\l,\L,\a) > 0$.
\end{lemma}

\begin{proof}
Let
\[
\tilde{v}(x) := \frac{[v-P](rx)}{r^{3+\a}}
\qquad\text{and}\qquad \tilde{A}(x) := A(rx).
\]
Then,
\[
\|\tilde{v}\|_{L^\infty(\overline{Q})} \leq 1.
\]
Since $P$ is an approximating polynomial for $v$ at zero, we have that
\[
\begin{cases}
\trace(\tilde{A}D^2\tilde{v}) = \tilde{h} &\text{in } Q\\
\pa_\nu \tilde{v} = 0 &\text{on }\Up_{\rm b} \cup \Up_{\rm l},
\end{cases}
\]
and from \eqref{eqn: conditions of remainder},
\[ 
|\tilde{h}| \leq C\vep.
\]

Now, consider the even reflections of $\tilde{A}, \tilde{v}$, and $\tilde{h}$ around the origin, which we denote by $\tilde{A}'', \tilde{v}''$, and $\tilde{h}''$.
It follows that (see, e.g., \cite[Proposition~4.1]{MS}),
\[
\trace(\tilde{A}'' D^2\tilde{v}'') = \tilde{h}'' \quad\text{in}\quad Q''.
\]
Since $\tilde{A}''$ are uniformly H\"{o}lder continuous and $\tilde{v}''$ and $\tilde{h}''$ are uniformly bounded in $Q''$, we deduce that $\tilde{v}''$ are locally uniformly H\"{o}lder continuous in $Q''$.
(Recall all these functions depend on our choice of $\vep$.)
By compactness, as $\vep$ converges to zero, we then find that, up to subsequences, $\tilde{v}''$ must converge uniformly in $Q''_\rho$ for every $\rho < 1$ to a function $v_0$ that is harmonic in $Q''$ and bounded by $1$.
Thus, since $\tilde{v}''|_{Q_\rho} = \tilde{v}$,
\begin{equation}
\label{eqn: first poly approx}
\|\tilde{v}-P_0\|_{L^\infty(\overline{Q}_\rho)}  \leq \|\tilde{v} - v_0\|_{L^\infty(\overline{Q}_\rho)}  + \|v_0 - P_0\|_{L^\infty(\overline{Q}_\rho)} \leq C\vep + C\rho^{4} \leq \frac{2}{3}\rho^{3+\a}
\end{equation}
if $\vep, \rho > 0$ are chosen sufficiently small.
Here, $P_0$ is the harmonic degree three Taylor polynomial of $v_0$ at the origin.
Furthermore, $\pa_1 P_0(0,\cdot) = \pa_2 P_0(\cdot,0) \equiv 0$ by symmetry.
Hence, $P_0$ has no linear or cubic part, no mixed two degree part, and the remaining two (pure) second degree coefficients of $P_0$ are such that
\begin{equation}
\label{eqn: coeff for harmonic poly approx}
0 = 2(p_0)_{2,1} + 2(p_0)_{2,2}
\end{equation}
Rescaling, we determine that
\[
\|v - P - r^{3+\a}P_0(\cdot/r)\|_{L^\infty(\overline{Q}_{\rho r})} \leq \frac{2}{3}(\rho r)^{3+\a}.
\]

Unfortunately, the polynomial
\[
P(x) + r^{3+\a}P_0(x/r)
\]
is not necessarily approximating for $v$ at zero.
To make it approximating, we want to replace $P_0$ with a polynomial $\bar{P}_0$ whose coefficients satisfy
\begin{equation}
\label{eqn: corrected poly coeff}
\begin{cases}
0 = 2(\bar{p}_0)_{2,1} + 2(\bar{p}_0)_{2,2}\\
0 = 6(\bar{p}_0)_{3,1} + r2(\bar{p}_0)_{2,1}\pa_1 a_1(0) + r2(\bar{p}_0)_{2,2}\pa_1 a_2(0)\\
0 = 6(\bar{p}_0)_{3,2} + r2(\bar{p}_0)_{2,1}\pa_2 a_1(0) + r2(\bar{p}_0)_{2,2}\pa_2 a_2(0).
\end{cases} 
\end{equation}
Subtracting \eqref{eqn: corrected poly coeff} and \eqref{eqn: coeff for harmonic poly approx}, we see that the coefficients for $P_0 - \bar{P}_0$ solve the system \eqref{eqn: corrected poly coeff} with left-hand side
\[
\begin{cases}
d_0 \hspace{-0.25cm}& = 0\\
d_{1,1} \hspace{-0.25cm}&= r2(p_0)_{2,1}\pa_1 a_1(0) + r2(p_0)_{2,2}\pa_1 a_2(0)\\
d_{1,2} \hspace{-0.25cm}&=  r2(p_0)_{2,1}\pa_2 a_1(0) + r2(p_0)_{2,2}\pa_2 a_2(0).
\end{cases}
\]
Thus, after choosing $(\bar{p}_0)_{2,1} = (p_0)_{2,1}$, since $\max_{m,j} |d_{m,j}| \leq Cr\vep$, it follows that $\bar{P}_0$ can be found so that
\[
\|P_0 - \bar{P}_0\|_{L^\infty(Q'')} \leq C\vep.
\]

Finally, replacing $P_0$ with $\bar{P}_0$ in \eqref{eqn: first poly approx}, we obtain the desired conclusion.
Also, observe that
\[
\bar{P}(x) := P(x) + r^{3+\a}\bar{P}_0(x/r)
\]
is such that
\[
\|\bar{P} - P\|_{L^\infty(Q''_r)} \leq Cr^{3+\a},
\]
as desired.
\end{proof}

After multiplying $v$ be a small constant, the hypotheses of Lemma~\ref{lem: iteration} are satisfied with $P \equiv 0$ and $r = r_0$.
Provided $r_0$ is sufficiently small (depending only on $\l,\L$, and $\alpha$), we can iteratively apply Lemma~\ref{lem: iteration} with $r = r_0\rho^k$ to determine the existence of a limiting approximating polynomial $P^0$ for $v$ at zero such that
\[
\|P^0\| \leq C \qquad\text{and}\qquad \|v - P^0\|_{L^\infty(\overline{Q}_r)} \leq Cr^{3+\a} \quad\forall r \leq r_0.
\]

Repeating a similar procedure at every point in $x \in \overline{Q}_{3/4}$, we find approximating polynomials $P^x$ such that the above inequalities holds with the same constant $C > 0$, which implies that $v \in C^{3,\a}(\overline{Q}_{3/4})$, as desired.
\end{proof}

At this point, we might hope to prove higher order Schauder estimates and then bootstrap to show that $u \in C^{\infty}(\overline{Q})$.
Recalling \eqref{eqn: pLt of neumann MA}, the regularity of the coefficients of our equation is limited by the regularity of $\nab u^\star$ up the boundary of $Q$.
However, this strategy falls short at the next stage.
The system of equations governing the existence of an approximating polynomial is degenerate, the normal derivative condition is too restrictive.
A simple manifestation of this is seen by considering
\[
\Delta v = x_1 x_2 \quad\text{in}\quad Q
\qquad\text{and}\qquad 
\pa_1 v(0,x_2) = \pa_2v(x_1,0) \equiv 0.
\]
A solution to this equation cannot be $C^4(\overline{Q})$.
Indeed, if it were, taking $\pa_{12}$, we see that
\[
\pa_{1112}v(0) + \pa_{2221}v(0) = 1.
\]
Yet, from the boundary data, we have that
\[
\pa_{1112}v(0) + \pa_{2221}v(0) = 0.
\]
This is impossible.
An adaptation of this example shows that optimal transport maps from $Q$ to itself may not be $C^3(\overline{Q})$ for generic (smooth) densities.
% \footnote{\,We thank Connor Mooney for a conversation from which this example was born.}

\begin{theorem}
\label{thm: counter ex to C3 Q}
There exist $f, g: \overline{Q} \to \R$ two smooth densities bounded away from zero and satisfying the mass balance condition $\|f\|_{L^1(Q)} = \|g\|_{L^1(Q)}$ such that the optimal transport $T$ taking $f$ to $g$ is of class $C^{2,\a}(\overline{Q})$ for every $\a < 1$ but not $C^3(\overline{Q})$.
\end{theorem}

\begin{proof}
Let $f(x) := 1 + x_1x_2$ and $g(y) := 5/4$.
By Theorem~\ref{thm: C2a up to bdry Q}, a convex potential $u$ defining $T$ is of class $C^{3,\a}(\overline{Q})$ for all $\a < 1$.
Now, suppose, to the contrary, that $u \in C^4(\overline{Q})$.
Then, $u^\star \in C^4(\overline{Q})$, and the first equation in \eqref{eqn: pLt of neumann MA} becomes
\[
\pa_{11}u^\star + \frac{5}{4}\pa_{22}u^\star = -x_2\pa_1u^\star\pa_{11}u^\star.
\]
So, differentiating in the $\e_1$-direction and then in the $\e_2$-direction, we see that
\[
\pa_{1112}u^\star + \frac{5}{4}\pa_{2221}u^\star = -((\pa_{11}u^\star)^2 + 2x_2\pa_{11}u^\star\pa_{112}u^\star + \pa_1u^\star\pa_{111}u^\star + x_2\pa_{12}u^\star\pa_{111}u^\star + x_2\pa_1u^\star\pa_{1112}u^\star).
\]
By the boundary conditions in \eqref{eqn: pLt of neumann MA} and recalling that $\pa_1u^\star(0) = 0$, we deduce that
\[
0 = \pa_{11}u^\star(0).
\]
Yet, this is impossible: by \eqref{eqn: strong convexity},
\[
\pa_{11}u^\star(0) = \frac{1}{\pa_{11}u(0)} \geq \frac{1}{C} > 0.
\]
\end{proof}

\begin{remark}
Replacing $f$ and $g$ with $f_\vep(x) := 1 + \vep x_1x_2$ and $g_\vep(y) := 1 + \vep/4$, we obtain the same contradiction as above, but with densities arbitrarily close (in $C^\infty$) to $1$.
\end{remark}

%~~~~~~~~~~~~~~~~~~~~~~~~~~~~~~~~~~~~~~~~~~~~~~~~~~~~~~~~~%
\noindent{\bf Acknowledgments.}
I would like to thank Connor Mooney for some keen observations and suggestions.
Also, I'm grateful to Alessio Figalli for his encouragement and guidance.

%~~~~~~~~~~~~~~~~~~~~~~~~~~~~~~~~~~~~~~~~~~~~~~~~~~~~~~~~~%

%~~~THE BIBLIOGRAPHY~~~%


\begin{thebibliography}{99}

\bibitem{B}
	\newblock Y. Brenier, 
	\newblock {\em Polar factorization and monotone rearrangement of vector-valued functions},
	\newblock Comm. Pure Appl. Math. {\bf 44} (1991), no. 4, 365-417.

\bibitem{C1} L. A. Caffarelli,
	\newblock {\em The regularity of mappings with a convex potential},
	\newblock J. Amer. Math. Soc. {\bf 5} (1992), no. 1, 99-104.

\bibitem{C2} L. A. Caffarelli,
	\newblock {\em Boundary regularity of maps with convex potentials},
	\newblock  Comm. Pure Appl. Math. {\bf 45} (1992), no. 9, 1141-1151.

\bibitem{C3} L. A. Caffarelli,
	\newblock {\em Boundary regularity of maps with convex potentials II},
	\newblock  Ann. of Math (2) {\bf 144} (1996), no. 3, 453-496.

\bibitem{CF} S. Chen and A. Figalli,
	\newblock {\em Boundary $\vep$-regularity in optimal transportation}, 
	\newblock  Adv. Math. {\bf 273} (2015), 540-567.	

\bibitem{CJLPR}
	\newblock O. Chodosh et al.,
	\newblock {\em On discontinuity of planar optimal transport maps}, 
	\newblock  J. Topol. Anal. {\bf 7} (2015), no. 2, 239-260.

\bibitem{DF} G. De Philippis and A. Figalli,
	\newblock {\em  Partial regularity for optimal transport maps},
	\newblock  Publ. Math. Inst. Hautes \'{E}tudes Sci. {\bf 121} (2015), 81-112.

\bibitem{FK}
	\newblock A. Figalli and Y.-H. Kim, 
	\newblock {\em Partial regularity of Brenier solutions of the Monge-Amp\`{e}re equation},  
	\newblock Discrete Contin. Dyn. Syst. {\bf 28} (2010), no. 2, 559-565.
			
\bibitem{FL}
	\newblock A. Figalli and G. Loeper, 
	\newblock {\em $C^1$ regularity of solutions of the Monge-Amp\`{e}re equation for optimal transport in dimension two}, 
	\newblock Calc. Var. Partial Differential Equations {\bf 35} (2009), no. 4, 537-550.

\bibitem{J}	Y. Jhaveri,
	\newblock {\em Partial regularity of solutions to the second boundary value problem for generated Jacobian equations},
	\newblock arXiv:1609.09680.

\bibitem{MS}
	\newblock E. Milakis and L. E. Silvestre, 
	\newblock {\em Regularity for fully nonlinear elliptic equations with Neumann boundary data}, 
	\newblock Comm. Partial Differential Equations {\bf 31} (2006), no. 7-9, 1227-1252.			

\end{thebibliography}
\end{document}